\documentclass[11pt]{amsart}
\usepackage[utf8]{inputenc}
\usepackage[english]{babel}
\usepackage[top=1.5in, bottom=1.5in, left=1.25in, right=1.25in]{geometry}
\usepackage{amsmath,amsthm,amsfonts,amssymb,amsbsy}
\usepackage{bm,bbm,verbatim}
\usepackage{cite}
\usepackage{color}
\usepackage{dsfont}
\usepackage{epsfig}
\usepackage{enumitem}
\usepackage{esint}
\usepackage{mathrsfs}
\usepackage{lipsum}
\usepackage{varioref,graphicx,epstopdf,subfigure}
\usepackage{stmaryrd}
\usepackage{yfonts}

\newtheorem{thm}{Theorem}[section]

\newtheorem{lem}[thm]{Lemma}
\newtheorem{prop}[thm]{Proposition}
\newtheorem{cor}[thm]{Corollary}
\newtheorem{cond}{Condition}
\newtheorem{defn}[thm]{Definition}

\newtheorem{rem}[thm]{Remark}

\renewcommand{\a}{\textfrak{a}}

\newcommand{\Dd}{\mathscr{D}}

\newcommand{\E}{\mathscr{E}}

\newcommand{\f}{\mathbf{f}}
\newcommand{\F}{\mathscr{F}}
\newcommand{\g}{\mathrm{g}}

\newcommand{\Hh}{\mathscr{H}}

\renewcommand{\L}{\mathscr{L}}
\newcommand{\Ll}{\mathfrak{L}}
\newcommand{\M}{\mathbf{M}}
\newcommand{\Mm}{\mathcal{M}}

\renewcommand{\P}{\mathscr{P}}

\newcommand{\R}{\mathbb{R}}

\renewcommand{\u}{u_{\star}}

\newcommand{\V}{\mathscr{V}}

\newcommand{\e}{\varepsilon}
  
\newcommand{\curr}[1]{\llbracket {#1}\rrbracket}  
\newcommand{\esssup}{\mathop{\rm ess\,sup}}       
\newcommand{\ttau}{\boldsymbol{\tau}}       
\renewcommand{\ss}{\subset\subset}       

\renewcommand{\div}{{\rm div }}

\newcommand{\reg}{\mathop{\rm reg}}       
\newcommand{\spt}{\mathop{\rm spt}}       
\newcommand{\sing}{\mathop{\rm sing}}       

\usepackage{scalerel}[2014/03/10]
\usepackage[usestackEOL]{stackengine}
\def\dashint{\,\ThisStyle{\ensurestackMath{%
  \stackinset{c}{.2\LMpt}{c}{.5\LMpt}{\SavedStyle-}{\SavedStyle\phantom{\int}}}%
  \setbox0=\hbox{$\SavedStyle\int\,$}\kern-\wd0}\int}

\newcommand{\lelbow}{\mathbin{
\vrule height 1.6ex depth 0pt width 0.13ex
\vrule height 0.13ex depth 0pt width 1.3ex}}

\newlength{\bibitemsep}
\newlength{\bibparskip}
\let\oldthebibliography\thebibliography
\renewcommand\thebibliography[1]{%
  \oldthebibliography{#1}%
  \setlength{\parskip}{\bibitemsep}%
  \setlength{\itemsep}{\bibparskip}%
}

\fontfamily{lmtt}\selectfont
\addto\captionsenglish{}

\addto\captionsenglish{}

\begin{document}

\title{Continuity of minimizers to weighted least gradient problems}

\author{Andres Zuniga}
\address{Department of Mathematics, Indiana University, 831 E 3rd St, Bloomington, IN 47405, USA.} 
\email{ajzuniga@indiana.edu.}

\begin{abstract}
We revisit the question of existence and regularity of minimizers to the weighted least gradient problem with Dirichlet boundary condition
\[
\inf\left\{\int_{\Omega}\a(x)|Du|:\, u\in BV(\Omega),\; u|_{\partial\Omega}=g\right\},
\]
where $g\in C(\partial\Omega)$, and $\a\in C^2(\bar{\Omega})$ is a weight function that is bounded away from zero. Under suitable geometric conditions on the domain $\Omega\subset \R^n$, we construct continuous solutions of the above problem for any dimension $n\geq 2$, by extending the Sternberg-Williams-Ziemer technique~\cite{sternberg1992existence} to this setting of inhomogeneous variations. We show that the level sets of the constructed minimizer are minimal surfaces in the conformal metric $\a^{2/(n-1)}I_n$. This result complements the approach in~\cite{jerrard2018existence} since it provides a continuous solution even in high dimensions where the possibility exists for level sets to develop singularities. The proof relies on an application of a strict maximum principle for sets with area minimizing boundary established by Leon Simon in~\cite{simon1987strict}.
\end{abstract}

\maketitle


\par

\smallskip

\section{Introduction}

In this article we revisit the question of existence and regularity of solutions in higher dimensions to weighted least gradient problems subject to a Dirichlet boundary condition 
\begin{equation}\label{eqn:simpleLGP}
\inf\left\{\int_{\Omega}\a(x)|Du|:\, u\in BV(\Omega),\; u|_{\partial\Omega}=g\right\},
\end{equation}
where $g\in C(\partial\Omega)$, and $\a\in C^2(\bar{\Omega})$ is a weight function that is bounded away from zero. Existence, comparison and uniqueness results in all dimensions were recently established in~\cite{jerrard2018existence} over a general class of integrands that includes the present case, and the solution was shown to be continuous in dimensions $n\leq 7$. The restriction on dimension in~\cite{jerrard2018existence} is due to an appeal to the regularity theory of hypersurfaces minimizing parametric elliptic functionals of Almgren, Schoen and Simon~\cite{schoen1977regularity,schoen1982new}. The major thrust of this article is to establish such a continuity result for a minimizer of~\eqref{eqn:simpleLGP} in higher dimensions $n\geq 8$ as well, using a constructive argument along the lines of that used in~\cite{sternberg1992existence} for the standard case $\a\equiv 1$. 

Going back to the work of Bombieri, De Giorgi and Giusti in~\cite{bombieri1969minimal}, extensive studies of functions of least gradient have been carried out in different contexts. The majority of the existing results for least gradient problems study the case of Dirichlet boundary conditions (see for instance~\cite{moradifam2018uniqueness,jerrard2018existence,gorny2016planar,mazon2016euler}). Nonetheless, Neumann and other types of boundary conditions have been explored (cf.~\cite{moradifam2017least,nachman2016weighted,nachman2007conductivity}). In the recent years many authors have spent a significant effort to study weighted least gradient problems and further generalizations, due to its various applications to such areas as imaging conductivity problems, reduced models in superconductivity and superfluidity, models for a description of landsliding, and relaxed models in the theory of elasticity and in optimal design, among others. A list of important investigations in these directions can be found in~\cite{athavale2017weighted,baldo2013vortex,gorny2016planar, gorny2017special, hoell2014current, jerrard2018existence, mazon2016euler, kohn1987constrained,moradifam2016existence, moradifam2017least, nashed2011structural,ionescu2005generalized, spradlin2014not,sternberg1992existence,sternberg1993dirichlet,
moradifam2012conductivity,moradifam2018uniqueness,nachman2007conductivity,nachman2009recovering,nachman2011current,nachman2016weighted}. In addition, the time dependent notion of total variation flow has proved to be useful in image processing including denoising and restoration, see for example~\cite{belletini2002total,andreu2004total,andreu2005total,moll2005anisotropic,caselles2009anisotropic}. Further generalizations of least gradient problems in the metric space setting have been explored quite recently in~\cite{hakkarainen2015stability,korte2016notions,lahti2017domains}.
\smallskip
  
Let us now introduce the problem more precisely, and the main result of this article. Given $n\geq 2$ arbitrary, a bounded Lipschitz domain $\Omega\subset \R^n$, and a weight function $\a\in C^2(\bar{\Omega})$ satisfying the following non-degeneracy condition
\begin{equation}\label{ND}
\min_{\bar{\Omega}}\a\geq \alpha,
\end{equation} 
for some $\alpha\in (0,\infty)$, we deal with the study of minimizers of the weighted $\a$-variation functional over the set of $BV(\Omega)$ functions that coincide on the boundary with some data $g:\partial\Omega\to\R$ in the sense of $BV$-traces. That is,
\begin{equation}\label{aLGP}\tag{$\a$\,LGP}
\inf_{u\in BV_g(\Omega)}\int_{\Omega}\a(x)|Du|,
\end{equation}
where the admissible class is defined via
\begin{equation}\label{eqn:admissible}
BV_g(\Omega):=\{u\in BV(\Omega): \;\forall x\in\partial\Omega, \;\lim_{r\to 0}\, \esssup_{y\in\Omega\cap|x-y|<r} |u(y)-g(x)|=0\}.
\end{equation}
Here $BV(\Omega)$ denotes the class of functions of bounded variation in $\Omega$ (see \cite{giusti1984minimal}). 

Let us recall the notion of $\a$-variation of $u\in BV(\Omega)$ induced by the continuous function $\a:\Omega\to (0,\infty)$, uniformly bounded away from zero.  As introduced by Amar and Belletini in~\cite{amar1994notion}, the $\a$-variation of $u\in BV(\Omega)$ in $\Omega$ is given by
\begin{equation}\label{eqn:defavar}
\int_U\a(x)|Du|:=\sup\left\{\int_Uu\,\div Y\,dx:\; Y\in C^{\infty}_c(U;\R^n),\; |Y(x)|\leq \a(x)\;\; \forall x\in\Omega\right\}.
\end{equation}
This corresponds to the definition of $\phi$-variation of $u$ in~\cite{amar1994notion} for the choice of $\phi(x,\xi)=\a(x)|\xi|$, which is described in terms of the dual norm $\phi^0(x,\xi):=\sup\{\xi\cdot p:\phi(x,p)\leq 1\}$. In~\eqref{eqn:defavar} we have used the fact that $\phi^0(x,\xi)=|\xi|/\a(x)$ for such choice of an inhomogeneous, isotropic norm $\phi$. This notion gives rise to a Radon measure on $\R^n$ induced by $u$ that acts on Borel sets via $B\mapsto \int_{B}\a(x)|Du|$, called the $\a$-variation measure of $u$. By analogy, given any Caccioppoli set $E\subset\R^n$ (i.e.~set of finite perimeter, see~\cite{giusti1984minimal}) we can construct an $\a$-perimeter measure associated with $E$, which is the Radon measure that on any Borel set $B$ assigns the value
\[
\P_{\a}(E,B):=\int_B\a(x)|D\chi_E|,
\]
where $\chi_E$ is the characteristic function of $E$.

\smallskip

The main concern of this work is to establish the existence of a {\em continuous minimizer} of~\eqref{aLGP} even in the possible presence of singularities for the level sets of the solution, when continuous boundary data $g\in C(\partial\Omega)$ is considered and for a class of domains $\Omega$ satisfying suitable geometric conditions. We will require that $\Omega$ is a bounded Lipschitz domain with boundary $\partial\Omega$ satisfying a positivity condition on a sort of generalized mean curvature related to the weight function $\a$. This will be referred as the {\em barrier condition}, and the precise statement is 
\begin{cond}[Barrier condition]\label{cond:barrier}
For every $x_0\in\partial \Omega$ there exists $\e_0>0$ such that for all $\e\in(0,\e_0)$ if $V_*\subset\R^n$ is a minimizer of
\begin{equation}\label{eqn:geomCiii}
\inf\{\P_{\a}(W,\R^n): \,W\subset \Omega,\;(\Omega\setminus W)\subset B_{\e}(x_0)\},
\end{equation}
then  
\[
\partial V_*\cap\partial\Omega\cap B_{\e}(x_0)=\emptyset.
\] 
\end{cond}

\smallskip
The boundaries of such domains $\Omega$ are not locally $\a$-area minimizing with respect to interior variations (cf.~\cite{jerrard2018existence}). In fact, the latter implies that domains $\Omega$ satisfying the barrier condition must necessarily have connected boundary. It is worth noting that even for domains $\Omega$ satisfying the barrier condition~\eqref{eqn:geomCiii}, it has been recently pointed out by Spradlin and Tamasan in~\cite{spradlin2014not} that the existence of minimizers to
\[
\inf\left\{\int_{\Omega}|Du|:\, u\in BV(\Omega),\; u|_{\partial\Omega}=g\right\}
\]
may fail for some choices of discontinuous boundary data $g$. 

An existence and continuity result of minimizers was already established by Jerrard, Moradifam and Nachman in~\cite{jerrard2018existence} for a more general version of the least gradient problem
\begin{equation}\label{phiLGP}\tag{$\varphi$\,LGP}
\inf_{u\in BV_g(\Omega)}\int_{\Omega}\varphi(x,Du),
\end{equation}
for the admissible class given in~\eqref{eqn:admissible}, $g\in C(\partial\Omega)$, and a function $\varphi(x,\xi)$ that, among other properties is convex, continuous, and 1-homogeneous with respect to the $\xi$-variable. They prove existence and comparison results (uniqueness) for~\eqref{phiLGP} valid in all dimensions $n\geq 2$ for domains $\Omega$ satisfying a barrier condition suited to a general class of inhomogeneous anisotropic $\varphi$-perimeter functionals. In contrast, their regularity theorem established for~\eqref{phiLGP}, under sharp conditions, is valid in low dimensions $n=2,3$ only. In a related work, Moradifam has argued that the structure of the level sets of minimizers to~\eqref{phiLGP} are determined by a divergence free vector-field (see~\cite{moradifam2016existence} for a precise statement).

Despite the dimensionality restriction of the regularity result in~\cite{jerrard2018existence}, it is nonetheless the case that 
when~\eqref{aLGP} is considered, i.e. $\varphi(x,\xi)=\a(x)|\xi|$, their result applies up to dimension $n\leq 7$ by virtue of the regularity theory of minimal hypersurfaces, with respect to an area functional induced by a Riemannian metric (see Remark 4.8 in \cite{jerrard2018existence} and references therein). In light of this, a major thrust of the present paper is to establish such a continuity result for a minimizer of~\eqref{aLGP} in higher dimensions $n\geq 8$.  

The approach we will adopt in this article consists of applying the Sternberg-Williams-Ziemer program in~\cite{sternberg1992existence} to construct continuous minimizers of the weighted least gradient problem subject to a Dirichlet boundary condition. In fact, a secondary reason for this investigation has been to determine whether this technique carries over to the setting of weighted least gradient problems. Their method is based on the co-area formula and on an auxiliary geometric variational problem to identify the level sets of such minimizers. Indeed, in~\cite{bombieri1969minimal} it was shown that the superlevel sets of a continuous function of least gradient are area-minimizing, that is, the characteristic functions of those sets are functions of least gradient. Conversely, the authors in~\cite{sternberg1992existence} proved the existence and continuity of a function of least gradient for every dimension $n\geq 2$, by explicitly constructing each of its superlevel sets in such a way that they are area-minimizing and reflect the boundary condition,  as long as two geometric conditions of $\partial\Omega$ are satisfied, referred as a weak non-negative mean curvature condition and the assumption that $\partial\Omega$ is not locally area-minimizing with respect to interior set variations. Their proof relies, among other things, on a strict maximum principle for area-minimizing sets established by Simon in~\cite{simon1987strict}.

In adapting the approach of~\cite{sternberg1992existence}, we concentrate our efforts in establishing a type  of maximum principle for sets that minimize the weighted $\a$-perimeter $\P_{\a}$ in $\Omega$, so as to ensure the {\em strict separation} of level sets of the candidate of a minimizer to~\eqref{eqn:simpleLGP}, from which the continuity of this minimizer will follow; see Theorem~\ref{thm:setmaxpple}. This maximum principle generalizes the corresponding one in~\cite{sternberg1992existence} for sets that minimize the standard area measure ($\a\equiv 1$). In this respect, the regularity assumed on the weight function in the present article, $\a\in C^2(\bar{\Omega})$, is required in both Proposition~\ref{prop:strongmaxpple} and Corollary~\ref{cor:simonsetmaxpple} (maximum principle for hypersurfaces with smooth, and non-smooth contact point). Moreover, the same regularity of $\a(x)$ is needed in our construction, independently, for a result about partial regularity of $\a$-perimeter minimizing sets due to Schoen and Simon~\cite{schoen1977regularity}, which we state in~\eqref{eqn:regularitySSA} (see (9) in~\cite{schoen1977regularity} for the regularity assumption in general.) In contrast, under the regularity assumption $\a\in C^{1,1}(\bar{\Omega})$ and positivity $\a>0$ of the weight, the authors in~\cite{jerrard2018existence} established uniqueness (for all $n\geq 2$) and continuity (up to $n\leq 7$) of minimizers to~\eqref{eqn:simpleLGP} for rougher weight functions and continuous boundary data $g\in C(\partial\Omega)$. In fact, they prove that the regularity  $\a\in C^{1,1}(\bar{\Omega})$ is sharp, in the sense that uniqueness of minimizers breaks.  

\smallskip

We now state the main result of this article.
\begin{thm}\label{thm:main1}
For any $n\geq 2$, let $\Omega\subset\R^n$ be a bounded Lipschitz domain with boundary satisfying the barrier condition~\eqref{eqn:geomCiii} and let $\a\in C^2(\bar{\Omega})$ be a non-degenerate weight function, in the sense of~\eqref{ND}. Then for any boundary data $g\in C(\partial\Omega)$ there exists a minimizer $\u$ to~\eqref{aLGP} which is moreover a continuous function, $\u\in C(\bar{\Omega})$. Furthermore, the superlevel sets of $\u$ minimize the weighted perimeter measure $\P_{\a}(\cdot,\Omega)$ with respect to competitors meeting the boundary conditions imposed by $g$ on $\partial\Omega$.  
\end{thm}

The paper is organized as follows. In \S2 we review some basic facts about the $\a$-variation functional and we comment on key aspects of the regularity theory for sets minimizing the $\a$-perimeter measure. 

In \S3 we establish a strict maximum principle for sets whose boundary minimize the $\a$-area, cf.~Theorem~\ref{thm:setmaxpple}. This will be done in two steps. We first address the case where the boundary sets can be locally represented as $C^2$-hypersurfaces, cf.~Proposition~\ref{prop:strongmaxpple}. The remaining case, where the hypersurfaces contain singularities, has been resolved by Leon Simon~\cite{simon1987strict} in the context of co-dimension one rectifiable currents which minimize mass. We proceed to review the concepts behind such mathematical objects in geometric measure theory, specializing in the context of Riemannian manifolds. In particular, in the setting where $\R^n$ is endowed with the metric $\g(x)=\a^{2/(n-1)}(x)\delta_{ij}dx^idx^j$ we can identify the $\a$-perimeter measure of a set $E$ (i.e~the $\a$-area of $\partial E$) with the mass of a current $\partial\curr{E}$, cf.~Theorem~\ref{thm:massequiv}. This fact allows us to apply the aforementioned result in~\cite{simon1987strict}.

The construction of the minimizer $u_{\star}$ takes place mainly in \S4 and \S5, where we introduce 
the collection of sets $\{E_t:t\in g(\partial\Omega)\}$, cf.~Proposition~\ref{prop:existence}. Here  $\overline{E_t\cap\Omega}$ will correspond, up to a $\Hh^n$-negligible set, to the $t$-superlevel set of $u_{\star}$ (Theorem~\ref{thm:constructionsoln}.) Subsequently, two key geometric properties of this collection are established, namely, the consistency of $\partial E_t$ with the boundary values at $\partial\Omega$, cf.~Lemma~\ref{lem:boundary}, and the strict separation of the sets $\{E_t\}$, cf.~Lemma~\ref{lem:strictcont}. The latter property is a consequence of Theorem~\ref{thm:setmaxpple}.

The candidate $u_{\star}$ for a minimizer of the Dirichlet problem~\eqref{aLGP} is introduced in~\S6. The admissibility and continuity properties of $u_{\star}$ follow from the properties obtained in \S5, cf.~Theorem~\ref{thm:constructionsoln}. Lastly, we argue the minimality of $u_{\star}$ in~\eqref{aLGP}, thus completing the proof of Theorem~\ref{thm:main1}.

\vskip.1in\noindent
{\bf Acknowledgments.} The author wishes to thank Peter Sternberg, Robert Jerrard, and the referee for their valuable comments and suggestions. The author was partially supported by the Hazel King Thompson fellowship from the Department of Mathematics at Indiana University.
\medskip

\section{Notation and preliminaries }\label{sec:notation}

Let us write $B_r(x)$ for the open Euclidean ball centered at $x\in\R^n$ of radius $r>0$, and we abbreviate $B_R:=B_R(0)$, unless otherwise specified. The notation $B'_r(x')$ will be reserved for balls in $\R^{n-1}$ centered at $x'\in\R^{n-1}$, where we will consistently write $x=(x',x'')\in\R^{n-1}\times\R$ for points $x\in\R^n$. With a slight abuse of notation we let $|\cdot|$ refer to the Euclidean distance between points in $\R^n$, and also to the Lebesgue measure in $\R^n$. In addition, $\Hh^{\alpha}$ corresponds to the $\alpha$-dimensional Hausdorff measure of $(\R^n,|\cdot|)$. Throughout, we will primarily employ $\Hh^{n-1}$. On the other hand, given a set $E\subset \R^n$, $E^i$ denotes the topological interior of $E$, $\bar{E}$ denotes the topological closure of $E$, and $\partial E$ denotes its topological boundary. Also, the notation $E\subset\subset F$ refers to the containment $\bar{E}\subset F^i$. We recall the measure-theoretic boundary of $E$,
\[
\partial_M E:=\{x\in\R^n:\; 0<\overline{\Theta}(E,x)\}\cap \{x\in\R^n:\underline{\Theta}(E,x)<1\},
\]
where 
\[
\overline{\Theta}(E,x):=\limsup_{r\to 0^+}|E\cap B_r(x)|/|B_r(x)|,\quad\underline{\Theta}(E,x):=\liminf_{r\to 0^+}|E\cap B_r(x)|/|B_r(x)|
\]
are the upper and lower densities of $E$ at $x$, respectively. Moreover, the {\em reduced boundary} of $E$ is the set $\partial^*E=\{x:\nu_E(x)\text{ exists}\}$ where $\nu_E(x)$ is the so called {\em measure-theoretic normal} of the set $E$, defined as the unique vector $\nu\in\R^n$ satisfying
\[
\overline{\Theta}(\{y: (y-x)\cdot \nu>0,\, y\in E\},x)=0,\quad \overline{\Theta}(\{y: (y-x)\cdot \nu<0,\,y\notin E\},x)=0.
\]
It is well-known that 
\begin{equation}\label{eqn:bdryinclusion}
\partial^*E\subset \partial_ME\subset \partial E.
\end{equation}
Moreover, $E$ is of finite perimeter if and only if $\Hh^{n-1}(\partial_ME)<\infty$; and in this case 
\[
\P(E,\Omega)=\Hh^{n-1}(\Omega\cap \partial_ME)=\Hh^{n-1}(\Omega\cap\partial^*E),
\]
cf. \cite{federer1969geometric}.
Throughout, we employ the {\em measure-theoretic closure} to represent the equivalence class of sets of finite perimeter, which differ only up to sets of $\Hh^n$-measure zero. With this convention, we let
\begin{equation}\label{eqn:convention}
x\in E \;\iff\; \overline{\Theta}(E,x)>0.
\end{equation}
It can be shown using this convention~\eqref{eqn:convention} that $\overline{\partial^*E}=\partial E$, cf. \cite[Thm.~4.4]{giusti1984minimal}. 

Suppose that $\mu$ is a Radon measure in a locally compact topological space $M$, and that $f\in L^1_{loc}(M,\mu)$. Here we adopt a notation already introduced in \cite{evans1992measure,simon1983lectures}, where we denote by $\mu\lelbow f$ the Radon measure acting on Borel sets of $M$ via
\begin{equation}\label{eqn:defelbow}
\mu\lelbow f(A):=\int_Af(x)\,d\mu(x).
\end{equation}
In addition $\lelbow$ stands for the restriction operation of a measure over a measurable set, in which case both relate by means of $\mu\lelbow A=\mu\lelbow \chi_A$.
\medskip

We continue this section by reviewing some basic properties of functions of bounded $\a$-variation and of sets of finite $\a$-perimeter. The results we revisit now hold for a broader class of weights which are continuous, $\a\in C$, and bounded away from zero, $\a>0$. For a proof of these facts we refer the reader to \cite{amar1994notion,jerrard2018existence}. 

The theory of inhomogeneous (and anisotropic) variations rests upon the following integral representation formula.
\begin{prop}[\!\!{\cite[Prop.~7.1]{amar1994notion}}] For any $u\in BV_{loc}(\R^n)$ and a bounded Borel set $B$,
\begin{equation}\label{eqn:integraldefn}
\int_B\a(x)|Du|=\int_B\a(x)|\sigma^u(x)|\,d|Du|(x),
\end{equation}
where $\sigma^u:=dDu/d|Du|\in\R^n$ denotes the Radon-Nikod\`ym density and $d|Du|$ is the total variation measure induced by $u$. Here we have $|\sigma^u|=1$ for $|Du|$-a.e $x\in B$ {\em(see \cite[\S5]{evans1992measure})}, so~\eqref{eqn:integraldefn} reduces to $\int_B\a(x)|Du|=\int_B\a(x)\,d|Du|(x)$.
\end{prop}
\noindent An immediate corollary of this proposition, using the characterization of the perimeter measure of Caccioppoli sets~\cite[\S4]{giusti1984minimal}, is the fact that for any Borel set $B$
\begin{align}
\P_{\a}(E,B)
&=\Hh^{n-1}\lelbow \a(\partial^*E\cap B)\label{eqn:charaperim},\\
&:=\int_{\partial^*E\cap B}\a(x)\, d\Hh^{n-1}(x).\notag
\end{align}

The integral representation formula \eqref{eqn:integraldefn} of the $\a$-variation, together with the Fleming-Rishel co-area formula for $BV$-functions imply a weighted version of the co-area formula
\begin{prop}[\!\!{\cite[Rem.~4.4]{amar1994notion}}]\label{prop:wcoarea}
If $u\in BV_{loc}(\R^n)$ and $B\subset\R^n$ is Borel, then
\[
\int_B\a(x)|Du|=\int^{+\infty}_{-\infty}\P_{\a}(\{u\geq t\},B)\,dt.
\]
\end{prop}

Furthermore, just like for the standard perimeter measure, the following inequality holds true as well for the $\a$-perimeter functional:
\begin{prop}[\!\!{\cite[Lem.~2.2]{jerrard2018existence}}]\label{prop:ineq}
For $B\subset\R^n$ Borel and $E_1,E_2\subset\R^n$ sets of locally finite $\a$-perimeter,  
\[
\P_{\a}(E_1\cup E_2,B)+\P_{\a}(E_1\cap E_2,B)\leq \P_{\a}(E_1,B)+\P_{\a}(E_2,B).
\]
\end{prop}

Essential in our development is the next lower semi-continuity property of the $\a$-perimeter, whose proof follows from a standard argument in the $BV$-theory. 
\begin{prop}\label{prop:lscaperim}
Let $U\subset\R^n$ be an open set, and $\{u_j\}\subset BV(U)$ a sequence of functions that converge in $L^1_{loc}(U)$ to a function $u$. Then $u$ has finite $\a$-variation in $U$, and moreover 
\[
\int_U\a(x)|Du|\leq \liminf_{j\to\infty}\int_U\a(x)|Du_j|.
\]
\end{prop}

Of particular importance to us are sets of finite $\a$-perimeter whose boundaries minimize the weighted area $\Hh^{n-1}\lelbow\a$, cf.~\eqref{eqn:defelbow}, which will be referred from now on as the $\a$-area measure. The notion of an $\a$-area minimizing set we use throughout, is the following

\begin{defn}
If $E$ is a set of locally finite $\a$-perimeter and $U$ is a bounded open set,
we say that $\partial E$ is $\a$-area minimizing in $U$ if
\begin{equation}\label{eqn:minaperim}
\P_{\a}(E,U)=\inf\{\P_{\a}(F,U):\,E\Delta F\subset\subset U\}.
\end{equation}
We say $\partial E$ is locally $\a$-area minimizing if~\eqref{eqn:minaperim} holds true for every choice of a bounded open subset $U$ of $\R^n$.
\end{defn}

We continue our preliminary discussion by highlighting some key aspects about the general strategy that we are going to adopt in the present article. A {\em continuous} function $u_{\star}$ of least $\a $-variation $\int_{\Omega}\a(x)|Du_{\star}|$ subject to Dirichlet data $g$, in the sense of~\eqref{eqn:simpleLGP}, will be constructed by first identifying the candidate $u_{\star}$ ``level set by level set." Inspired by the co-area formula Proposition~\ref{prop:wcoarea}, each superlevel set of $u_{\star}$ will correspond (up to a set of $\Hh^n$-measure zero) to a set of minimal $\a$-perimeter in $\Omega$ and meeting the boundary conditions on $\partial\Omega$ imposed by $g$. Roughly speaking, given a value of $t\in g(\partial\Omega)$ we are going to solve the problem of finding a set $E_t$ that minimizes the $\a$-perimeter $\P_{\a}(\cdot,\Omega)$ while meeting $\{g\geq t\}$ on $\partial\Omega$. Once the collection $\{E_t:t\in g(\partial\Omega)\}$ has been built, the continuity of $u_{\star}$ will be a consequence of a key ingredient, known as a strict maximum principle for $\a$-area minimizing sets; a separation $\partial E_t\cap\partial E_s\cap\Omega=\emptyset$ will be then obtained for any $s<t$. Such a maximum principle is local in nature, so we will address it in two different cases depending on whether the sets $\partial E_t$, $\partial E_s$ can be written as a $C^2$-hypersurfaces around a contact point.

As the discussion in the preceding paragraph suggests, the regularity of $\partial E_t$ will play a crucial role in our development just like in the standard (homogeneous) theory. Given an $n$-rectifiable set $E\subset\R^n$, we say that $\partial E$ is regular at a point $x\in\partial E$ if there is a ball $B\subset\R^n$ centered at $x$ such that $\partial E\cap B$ has a representation as a $C^2$-hypersurface, written in coordinates around the approximate tangent space at $x$ (see~\cite{simon1983lectures}). Let us write $\reg(\partial E):=\{x\in\partial E: \partial E\text{ is regular at }x \}$ and $\sing(\partial E):=\partial E\setminus\reg(\partial E)$. The building blocks to the regularity theory in the standard homogeneous setting, which play a similar role in our inhomogeneous setting,  are the tangent cones. More precisely, if $\partial E$ is $\a$-area minimizing in a bounded open set $U\subset \R^n$, then for each $x\in \partial E$ and each sequence $\lambda_j\to 0^+$ there exist a subsequence $\{\lambda_{j'}\}\subset \{\lambda_j\}$ and a Borel set $F$ of locally finite $\a$-perimeter so that $\partial F$ is $\a$-area minimizing in $U$, and if we denote the translation plus homothety $E_j:=\{y\in\R^n: x+\lambda_j(y-x)\in E\}$, it also follows that $\chi_{E_{j'}}\to \chi_{F}$ in $L^1_{loc}(\R^n)$. Here $C=\partial F$ is called \emph{tangent cone} to $E$ at $x$ (not unique a priori). Although it is not immediate, $C$ is a union of half-lines issuing from $x$ and thus is a cone. If $\bar{C}$ is contained in any hyperplane of $\R^n$ supported at $x$ then the tangent cone of $\partial E$ at $x$ is unique, and $\partial E$ is regular at $x$ (see~\cite[Thm.~I.1.2]{schoen1977regularity}). In other words, there exists $r>0$ so that $B_r(x)\cap \partial E$ is a $C^2$-hypersurface. See Remark~\ref{rem:regularity} in \S\ref{sec:maxpple} for additional comments. 

Another crucial tool in the proof of Theorem~\ref{thm:main1} is a singularity estimate due to Schoen, Simon and Almgren~\cite[Thm.~I.3.1, Cor.~I.3.2]{schoen1977regularity}. Although this result is quite general, we apply it in the context of sets whose boundary minimize the $\a$-area. It states that for a set $E$ with $\partial E$ minimizing the $\a$-area in a bounded open set $U$, there holds
\begin{equation}\label{eqn:regularitySSA}
\left\{\begin{array}{ll} \Hh^{n-3}(\sing(\partial E)\cap U)<\infty, & \text{if } n\geq 4\\[0.1cm] \sing(\partial E)\cap U=\emptyset, & \text{if }n\leq 3\end{array}.\right.
\end{equation}
It follows that $\reg(\partial E)$ is dense in $\partial E$. Estimate~\eqref{eqn:regularitySSA} rests upon a regularity assumption on the weight function $\a\in C^2$ (see (9) in~\cite{schoen1977regularity} for the precise condition in generality).

\section{A strict maximum principle for $\a$-perimeter minimizing sets}\label{sec:maxpple}

The aim of the current section is to present a general maximum principle for $\a$-perimeter minimizing sets, and to provide a full proof of it.  It states that when two sets are nested and touch at a point, and have $\a$-minimizing perimeter, then they should locally coincide.
\smallskip

More precisely, this result states that 

\begin{thm}[Maximum principle for $\a$-area minimizing sets]\label{thm:setmaxpple}
Let $E_1\subset E_2$ be $n$-rectifiable subsets in $\R^n$ where both $\partial E_1$, $\partial E_2$ minimize the $\a$-area in a bounded open set $U\subset \R^n$. If $x\in \partial E_1\cap\partial E_2\cap U$, then $\partial E_1$ and $\partial E_2$ agree in some neighborhood of $x$.
\end{thm}

This is an adaptation of the analogous Theorem 2.2 in~\cite{simon1987strict} to the particular setting of an inhomogeneous isotropic Riemannian metric. The contribution of such result, when compared to the classical strong maximum principle for minimal surfaces, is that it includes the case of hypersurfaces that may contain singularities, around which the hypersurface {\em cannot} be written as the graph of a function over the tangent plane based at the singularity point. In particular, the contact point $x$ mentioned above could potentially be a singular point for either of $E_1$ or $E_2$. This problematic situation in the context of minimal surfaces has been resolved by Leon Simon in a celebrated result, known as the {\em strict maximum principle for mass minimizing currents}, of fundamental importance in  geometric measure theory.  It deals with a more general situation than the one mentioned in the above theorem, where the main object of study are currents (cf.~\S3.2). This result states the following

\begin{thm}[\!\!\!{\cite[Thm.~1]{simon1987strict}}]\label{thm:simonmaxpple}
Let $U$ be an open set of a smooth $n$-dimensional oriented Riemannian manifold $N$.
Suppose $T_1$ and $T_2$ are integer multiplicity currents with $\partial T_1=0=\partial T_2$ in $U$,~$T_1$ and $T_2$ are mass-minimizing in $U$ and $\reg T_1\cap \reg T_2\cap U=\emptyset$. Then, 
\[
\spt T_1\cap \spt T_2\cap U=\emptyset.
\]
\end{thm} 
The main content of this theorem lies in the fact that $\sing T_1\cap\sing T_2\cap U=\emptyset$. Indeed, a previous work in~\cite{miranda1967sulle} and also in~\cite[\S37.10]{simon1983lectures} establishes $\sing T_1\cap \reg T_2\cap U=\emptyset$. The latter fact was subsequently proved in~\cite{solomon1989strong}, even without the minimizing hypothesis.
\smallskip

A direct consequence of Theorem~\ref{thm:simonmaxpple} is the next corollary for oriented boundaries of least area
\begin{cor}[\!\!{\cite[Cor.~1]{simon1987strict}}]\label{cor:simonsetmaxpple}
Adopting the notation of Theorem~\ref{thm:simonmaxpple}, let $T_1=(\partial\curr{E_1})\lelbow U$ and $T_2=(\partial\curr{E_2})\lelbow U$ be mass-minimizing currents in $U$, with $E_1\cap U\subset E_2\cap U$ and with $\spt T_1\cap U$ and $\spt T_2\cap U$ connected. Then either $T_1=T_2$, or~$\spt T_1\cap\spt T_2\cap U=\emptyset$. 
\end{cor}

It is ultimately this tool the one that allows us to push our continuity result to all dimensions $n\geq 2$, and in particular to the case $n\geq 8$ where $\a$-area minimizing sets cease to be smooth in general. In this situation, PDE techniques such as the strong maximum principle are no longer available (not even in the weak form). For details, the reader can compare to the continuity result in \cite{jerrard2018existence}.
\smallskip

The proof of Theorem~\ref{thm:setmaxpple} is rather technical and it uses different ingredients in geometric analysis, the main one being Corollary~\ref{cor:simonsetmaxpple}. The purpose of this section is to present a collection of results, along with their proofs, that will give a proof of Theorem~\ref{thm:setmaxpple}.

\subsection{A strict maximum principle for hypersurfaces}

Our first goal is to establish a strict maximum principle for hypersurfaces minimizing the $\a$-area functional, in the case where the intersection point is regular for both hypersurfaces. In fact, we prove a slightly stronger result; see Remark~\ref{rem:scopestrongmaxpple}. In this situation we have PDE techniques available, and the next result corresponds to a weak version of the Hopf maximum principle for quasilinear equations in divergence form. Before introducing this result, we need to discuss some preliminary concepts.

Given a Caccioppoli set $E$ and $x_0\in\partial^*E$ it is known that $\partial E$ can be locally represented as the graph of a non-negative $C^1$-function $u$ over a tangent plane at $x_0$, cf. \cite{giusti1984minimal}. If we write $(x',s)\in\R^{n-1}\times\R$ for the coordinates of $x\in\R^n$ and if $B'_r(x')$ denotes a ball in $\R^{n-1}$, then the latter fact shows that (up to isometries) there is a choice of $r=r(x_0)>0$ small and $u\in C^1(B'_r(x'_0);[0,\infty))$ in such a way that  
\[
\{(x',s):\, x'\in B'_r(x'_0),\; 0\leq s\leq u(x')\}\subset E,
\]
 and
\[
\partial E\cap B_r(x_0)=\{(x',u(x')):\,x'\in B'_r(x'_0)\}.
\]
In particular, the characterization~\eqref{eqn:charaperim} of the $\a$-perimeter measure together with the area formula (cf.~\cite[Thm.~4.1]{evans1992measure}) allows us to compute in local coordinates 
\begin{align}
\P_{\a}(E,B_r(x_0))
&=\int_{B'_r(x'_0)}\a(x',u(x'))\sqrt{1+|\nabla u(x')|^2}\,d\Hh^{n-1}(x') \label{eqn:aLocalPer}\\
&=:I_{\a}(u).\notag
\end{align}
Here and henceforth, we denote $\nabla:=\nabla_{x'}$ for the gradient in $\R^{n-1}$, whenever appropriate. Computing the first variation of $I_{\a}(u)$ in the direction of $\varphi$, $\delta I_{\a}(u)[\varphi]$, we obtain
\begin{equation}
\Mm_{\a} u(\varphi):=\int_{B'_r(x'_0)}
\left\{\frac{\a(x',u)\nabla u}{\sqrt{1+|\nabla u|^2}}\cdot\nabla\varphi
+\partial_s\a(x',u)\sqrt{1+|\nabla u|^2}\,\varphi\right\} dx'=0.
\end{equation}
This operator will be called the $\a$-minimal surface operator of $u$, acting on any test function $\varphi$ with compact support in $B'_{r}(x'_0)$. The last observation motivates the following notions

\begin{defn}\label{def:SolaMSE}
Let $W\subset \R^{n-1}$ be an open set.  A function $u\in C^1(W)$ is called a classical solution to the inhomogeneous $\a$-minimal surface equation {\em (or $\a$-MSE)} in $W$, if
\[
-\div_{x'}\left(\frac{\a(x',u)\nabla u}{\sqrt{1+|\nabla u|^2}}\right)+\partial_s\a(x',u)\sqrt{1+|\nabla u|^2} =0,\;\; \;\forall x'\in W.
\]
Also, $u\in C^1(W)$ is called a weak subsolution {\em (}weak supersolution{\em)} of the inhomogeneous $\a$-MSE in $W$ if
\[
\Mm_{\a}u(\varphi)
:=\int_{W}\left\{\frac{\a(x',u)\nabla u}{\sqrt{1+|\nabla u|^2}}\cdot\nabla\varphi
+\partial_s\a(x',u)\sqrt{1+|\nabla u|^2}\,\varphi\right\} dx'\leq 0 \quad (\geq 0),
\]
whenever $\varphi\in C^{1}_c(W)$ with $\varphi\geq 0$.
\end{defn}
We now present a maximum principle, in the favorable situation where the contact point is regular for the two touching hypersurfaces which minimize the $\a$-area functional.
\begin{prop}[Weak form of strict maximum principle for hypersurfaces]\label{prop:strongmaxpple}
Let $W$ be an open set in $\R^{n-1}$ and let $u_0, u_1\in C^1(W)$ be a weak supersolution and a weak subsolution of the inhomogeneous $\a$-MSE in $W$, respectively. Suppose that $u_1\geq u_0$ in $W$, while $u_1(x'_0)=u_0(x'_0)$ for some $x'_0\in W$. Then $u_1=u_0$ in some open neighborhood of $x'_0$ in $W$. 
\end{prop}

\begin{proof}[Proof Proposition~\ref{prop:strongmaxpple}]
Let us write $w(x'):=u_1(x')-u_0(x')$ and $u^t(x'):=u_0(x')+tw(x')$ for $t\in [0,1]$ and $x'\in W$. Let us observe that $u^0=u_0$ and $u^1=u_1$, in view of this notation. It will be convenient for computations to rewrite the $\a$-minimal surface operator acting on $\varphi\in C^{1}_c(W)$ with $\varphi\geq 0$, as
\[
\Mm_{\a}u(\varphi)=
\int_W(\f(x',u,\nabla u)\cdot\nabla\varphi+f(x',u,\nabla u)\varphi)dx',
\]
where for $x',p'\in\R^{n-1}$ and $s\in\R$  we let
\[
\f(x',s,p'):=\a(x',s)\frac{p'}{\sqrt{1+|p'|^2}},\quad f(x',s,p'):=\partial_s\a(x',s)\sqrt{1+|p'|^2}.
\]
Using the linearity of  $\varphi\mapsto \Mm_{\a}u(\varphi)$, we compute the difference $\Mm_{\a}u^1(\varphi)-\Mm_{\a}u^0(\varphi)$ by means of the chain rule
\[
\f(x',u^1,\nabla u^1)-\f(x',u^0,\nabla u^0)
=\int^1_0\frac{d}{dt}\f(x',u^t,\nabla u^t)\,dt
=\{a^{ij}(x')\partial_jw+b^i(x')w\}_{i=1,\ldots, n-1},
\]
with coefficients 
\begin{align*}
a^{ij}(x')
&:=
\int^1_0\left[D_{p'}\f(x',u^t,\nabla u^t)\right]_{ij}dt
=\int^1_0\a(x',u^t)
\frac{(1+|\nabla u^t|^2)\delta_{ij}-\partial_{i}u^t\partial_{j} u^t}{(1+|\nabla u^t|^2)^{3/2}}\,dt,
\\
b^i(x')&:=\int^1_0\left[\partial_s\f(x',u^t,\nabla u^t)\right]_{i}dt
=\int^1_0\frac{\partial_s\a(x',u^t)\partial_i u^t}{\sqrt{1+|\nabla u^t|^2}}\,dt.
\end{align*}
Similarly, one has
\[
f(x',u^1,\nabla u^1)-f(x',u^0,\nabla u^0)
=\int^1_0\frac{d}{dt}f(x',u^t,\nabla u^t)\,dt
=-(c^j(x')\partial_jw+d(x') w),
\]
where
\begin{align*}
c^j(x')&:=-\int^1_0\left[D_{p'}f(x',u^t,\nabla u^t)\right]_{j}dt
=\int^1_0-\partial_{s}\a(x',u^t)\frac{\partial_{j}u^t}{\sqrt{1+|\nabla u^t|^2}}\,dt,\\
d(x')&:=-\int^1_0\partial_sf(x',u^t,\nabla u^t)dt=\int^1_0-\partial_{ss}\a(x',u^t)\sqrt{1+|\nabla u^t|^2}\,dt.
\end{align*}
Since $u^1$ and $u^0$ are weak supersolution and subsolution of $\a$-MSE, respectively, we get
\begin{align*}
0&\leq \Mm_{\a}u^1(\varphi)-\Mm_{\a}u^0(\varphi)\\
&=\int_W\left\{
(\f(x',u^1,\nabla u^1)-\f(x',u^0,\nabla u^0))\cdot\nabla\varphi+
(f(x',u^1,\nabla u^1)-f(x',u^0,\nabla u^0))\varphi\right\}dx'\\
&= \int_W \{ (a^{ij}(x')\partial_jw+b^i(x')w)\partial_i\varphi-(c^j(x')\partial_jw+d(x')w)\varphi\}dx'\\
&=:\Ll(w,\varphi).
\end{align*}
Thus, $w$ is a weak supersolution of $Lw=0$, where $L$ corresponds to the linear operator in divergence form 
\[
Lw:=\partial_{i}\Big(a^{ij}(x')\partial_{j}w+b^i(x')w\Big)+c^j(x')\partial_jw+d(x')w.
\]
We verify that $L$ is uniformly elliptic in $W$ and that it has bounded coefficients. Indeed, for each neighborhood $V\ss W$ of $x'_0$ there exists $K>0$ so that $\sup_{t\in[0,1]}\|u^t\|_{C^1(\bar{V})}\leq K$, since $u_1,u_0\in C^1(W)$. In particular, $\a\in C^2$ implies that $\exists K_{\a}>0$ so that $\|\a\|_{C^2(\bar{W}\times [-K,K])}\leq K_{\a}$.  The assumed non-degeneracy \eqref{ND} yields the uniform ellipticity: for $\xi'\in\R^{n-1}$ and $x'\in V$,  $a^{ij}(x')\xi'_i\xi'_j\geq \alpha|\xi'|^2/(1+K^2)^{3/2}$. In addition, there exist constants $\Lambda(K_{\a},K),\nu(K_{\a},K)$ so 
\[
\sup_{x'\in \bar{V}}\sum_{i,j}|a^{ij}(x')|^2\leq \Lambda \;\;\text{ and }\;\;
\sup_{x'\in \bar{V}}\sum_i|b^i(x')|^2+\sum_j|c^j(x')|^2+|d(x')|\leq \nu^2.
\]
By invoking the \emph{weak Harnack inequality}\cite[Thm.~8.18]{gilbarg1998elliptic} on the non-negative supersolution $w$ of $Lw=0$, we get the existence of $C>0$ depending on $K, \alpha,\Lambda, \nu, n$, in such a way that 
\[
\e^{-n/q}_0\biggl(\int_{B'_{2\e_0}(x'_0)}|w|^q\biggr)^{1/q}\leq C\inf_{B'_{\e_0}(x'_0)}w=0,
\]
for every $q\in (1,n/(n-2))$ and $\e_0$ with $B'_{4\e_0}(x'_0)\subset V$. Therefore $w\equiv 0$ on $B'_{\e_0}(x'_0)$.
\end{proof}

\begin{rem}\label{rem:scopestrongmaxpple}
The strict maximum principle in Proposition~\ref{prop:strongmaxpple} remains valid for hypersurfaces which may not minimize the $\a$-area  necessarily. Furthermore, these hypersurfaces need not be critical for the $\a$-area functional~\eqref{eqn:aLocalPer}, but rather supercritical and subcritical around the contact point, in the sense that the associated functions $u_0$ and $u_1$ in Proposition~\ref{prop:strongmaxpple} are supersolution and subsolution of the $\a$-MSE, respectively.
\end{rem}

\subsection{Duality between weighted variations and mass of currents}

The main tool in the proof of our desired Theorem~\ref{thm:setmaxpple} consists of a broader kind of maximum principle, for mass-minimizing currents established in~\cite{simon1987strict}. With the purpose of applying this maximum principle to our setting (cf.~Corollary~\ref{cor:simonsetmaxpple}) it will be then necessary to make the identification between the {\em $\a$-perimeter measure} of a Caccioppoli set $E$, with the notion of {\em mass} of the co-dimension one rectifiable current $\partial\curr{E}$, in the sense of Federer~\cite{federer1969geometric}. A reader well versed in geometric measure theory may consider this identification rather clear, however, for completeness and readability of the article we give a quick overview of some needed concepts in geometric measure theory.

Throughout, $M$ denotes an oriented complete smooth Riemannian manifold of dimension $n$ endowed with a $C^2$-metric $\g$ (cf.~Remark~\ref{rem:regularity}), $U$ denotes a non-empty open subset of $M$, and $l\in\{1,\ldots,n\}$. The $\alpha$-dimensional Hausdorff measure induced by the geodesic distance $d_{\g}$ in $M$ will be written $\Hh^{\alpha}_{\g}$. The space of smooth differential $l$-forms in $M$ is denoted by $\Omega^l(M)$. In particular, we are interested in those forms that have compact support in $U$,
\[
\Dd^{l}(U):=\{\omega\in\Omega^{l}(M):\,\spt \omega\subset U\}.
\]
This set is equipped with the standard topology of convergence on compact subsets of $M$.
An $l$-dimensional current in $U$ is defined as a continuous linear functional over $\Dd^{l}(U)$. The set of all $l$-dimensional currents over $U$ will denoted by $\Dd_l(U)$. Following \cite[\S26]{simon1983lectures} we introduce in the Riemannian setting a particularly useful class of currents, namely, the class integer multiplicity rectifiable currents (cf.~\cite{federer1969geometric}).
\smallskip

\begin{defn}[Integer multiplicity rectifiable current]\label{def:imc}
Let $U$ be an open set of $(M,{\g})$. If $T\in \Dd_l(U)$, we say $T$ is an integer multiplicity rectifiable current if it can be expressed as
\begin{equation}\label{eqn:integermultcurrent}
T(\omega)=\int_N\langle w,\xi\rangle_x\,\theta(x)\,d\Hh^l_{\g}(x)\;\;\text{ for }\; \omega\in\Dd^l(U),
\end{equation}
where $N$ is an $\Hh^l_{\g}$-measurable countably $l$-rectifiable subset of $U$, $\theta$ is an integer valued function in $L^1_{loc}(N,\Hh^l_{\g})$, $\xi$ is a $\Hh^l_{\g}$-measurable, oriented unit $l$-vector field on $N$, and $\langle\cdot ,\cdot\rangle_x$ corresponds to the dual pairing between the spaces $\Lambda^l(T_xN)$ of $l$-covectors and $\Lambda_l(T_xN)$ of $l$-vectors, in the approximate tangent space $T_xN$.
\end{defn}

In the case $T$ is as in \eqref{eqn:integermultcurrent}, we write $T=:\ttau(N,\theta,\xi)$ and call $\theta$ the multiplicity of $T$, $\xi$ the orientation of $T$.  A central role in the maximum principle in~\cite{simon1987strict} is played by currents induced by sets E, obtained by integration over $E$ of smooth $l$-forms in $M$ with compact support. That is,
\begin{defn}[Current induced by a set]\label{def1}
Let $E$ be a countably rectifiable subset of $M$ which is $\Hh^l_{\g}$-measurable for some $l\in\{1,\ldots,n\}$. We define the $l$-dimensional current $\curr{E}$, acting on $\eta\in\Dd^{l}(M)$ via
\begin{equation}\label{eqn:1.1}
\curr{E}(\eta):=\int_E\eta,
\end{equation}
where the integration is generally defined in the sense of~{\em \cite[\S11.1, \S11.7]{simon1983lectures}}.
\end{defn}
We point out that~\eqref{eqn:1.1} yields a sensible definition of $\curr{E}$ even when the ambient space $M$ is not equipped with a metric $\g$. Nonetheless, in the Riemannian setting it is convenient from the point of view of geometric measure theory, to rewrite this expression in terms of a measure arising from the metric of $M$. Thus, we have an alternative characterization for currents arising from appropriate sets, as given in

\begin{prop}\label{prop:equiv}
For $E$ as in Definition~\ref{def1}, we have that $\curr{E}$ is a multiplicity one rectifiable current, $\curr{E}=\ttau(E,1,\xi)$, where $\xi$ is the oriented unit $l$-vector field on $E$ inherited from $M$. In other words,
\[
\curr{E}(\eta)=\int_E\langle\eta,\xi\rangle_x\, d\Hh^l_{\g}(x)\;\;\text{ for }\;\eta\in\Dd^l(M).
\]
\end{prop}

We now continue by recalling the notion of {\em boundary} of currents. Motivated by the classical Stoke's theorem, we are led by~\eqref{eqn:1.1} to quite generally define the boundary $\partial T$ of an $l$-current $T\in\Dd_l(M)$ by $\partial T(\omega):=T(d\omega)$ for $\omega\in\Dd^{l-1}(M)$. Finally we review the concept of {\em mass of a current}, in the sense of~\cite[\S26.4]{simon1983lectures}. Again motivated by the special case $T=\curr{E}$ as in~\eqref{eqn:1.1}, we define the mass of the current $T$ with respect to the metric $\g$, denoted $\M_{\g}(T)$, for $T\in\Dd_l(M)$ by
\begin{equation}\label{eqn:defnmass}
\M_{\g}(T):=\sup\left\{T(\omega):\,\omega\in\Dd^l(M),\;\; \sup_{x\in M}|\omega_x|_{\g}\leq 1\right\}
\end{equation}
(so that $\M_{\g}(T)=\Hh^n_{\g}(E)$ if $T=\curr{E}$.) More generally, for any open $U \subset M$ we define the mass of the current $T$ restricted to $U$ with respect to the metric $\g$, by
\begin{equation}\label{eqn:defnmass2}
\M_{U,\g}(T):=\sup\left\{T(\omega):\,\omega\in\Dd^l(M),\; \spt\omega\subset U,\;\; \sup_{x\in U}|\omega_x|_{\g}\leq 1\right\},
\end{equation}
where in~\eqref{eqn:defnmass}-\eqref{eqn:defnmass2} we adopt the convention where the norm of an $l$-covector $\omega_x$, $|\omega_x|_{\g}$, is defined as the dual norm to $\g$, that acts on the space of $l$-vectors.
\smallskip
   
We are ready to discuss the main result of this subsection, in which we identify the $\a$-perimeter of a set $E\subset \R^n$ with the mass of the current $\curr{E}$ in the ambient manifold $M=\R^n$, with respect to a certain metric $\g$ depending on the weight function $\a$. This will allow us to invoke the maximum principle in~\cite{simon1987strict} to our development.   
 
\begin{thm}\label{thm:massequiv}
Consider $\R^n$ endowed with the metric $\g=\a^{\sigma}\bar{\g}$, conformal to the standard Euclidean $\bar{\g}_x=\delta_{ij}dx^idx^j$ with $\sigma\neq 0$. If $E$ is a Caccioppoli set that is furthermore an $n$-rectifiable Borel set, then there holds
\[
\M_{U,\a^{\sigma}\bar{\g}}(\partial\curr{E})=\int_{U}\a^{(n-1)\sigma/2}(x)\, d|D\chi_E|(x),
\]
for any connected open set $U\subset \R^n$. In particular, the choice $\sigma=2/(n-1)$ yields
\[
\M_{U,\a^{2/(n-1)}\bar{\g}}(\partial\curr{E})=\P_{\a}(E,U).
\]
\end{thm}
\vskip 1pt
\begin{proof}[Proof of Theorem~\ref{thm:massequiv}]
Clearly $[\partial/\partial x^1,\ldots, \partial/\partial x^n]$ is a global orientation of the ambient manifold $\R^n$, where $(x^1,\ldots, x^n)$ denote the standard coordinates. The countably $n$-rectifiable set $E$ is $\Hh^n_{\g}$-measurable since it is Borel in $\R^n$, hence it induces the $n$-current $\curr{E}:=\ttau(E,1,\xi)$ as in Proposition~\ref{prop:equiv} (with $l=n$), where $\xi$ is the unit standard orientation in $E$ induced from $\R^n$. 

According to the definition of mass in~\eqref{eqn:defnmass2} and of boundary current, $\M_{U,\g}(\partial\curr{E})$ is the supremum of $\curr{E}(d\omega)$ for any $\omega\in\Dd^{n-1}(\R^n)$ with $\spt\omega\subset U$ and $\sup_{x\in U}|\omega(x)|_{\g}\leq  1$. On the other hand, it is known that the Hodge star operator $*:\Omega^1(U)\to\Omega^{n-1}(U)$ is a linear isometry. Thus, $\omega\in\Omega^{n-1}(U)$ can be written $\omega=*\Upsilon$ for $\Upsilon\in \Omega^1(U)$, and furthermore 1-forms admit a vector field representation by means of the musical isomorphism: $\Upsilon=X^{\flat}$, which is an isometry by construction of the flat operator $\flat: C^{\infty}(U;\R^n)\to \Omega^1(U)$. Let us now remark that $\spt X\subset\subset U$ iff $\spt(*X^{\flat})\subset\subset U$. Indeed, we can write explicitly in coordinates
\begin{align*}
*X^{\flat}=*\left(\sum^n_{j=1}(\sum^n_{i=1}g_{ij}X^i)\frac{\partial}{\partial x_j}\right)
=\sum^n_{j=1}(\sum^n_{i=1}g_{ij}X^i)(-1)^{j-1}dx^1\wedge\ldots\wedge\widehat{dx^j}\wedge\ldots\wedge dx^n,
\end{align*}
where $\widehat{dx^j}$ is used to denote that the term $dx^j$ is missing in the product above. 
In light of these observations it follows that $|(*X^{\flat})(x)|_{\g}=|X^{\flat}(x)|_{\g}=|X(x)|_{\g}$, and therefore
\begin{equation}\label{eqn:intermass}
\M_{U,\g}(\partial\curr{E})=\sup\left\{\curr{E}(d*X^{\flat}): X\in C^{\infty}(U;\R^n),\; \spt X\subset U,\; \sup_{x\in U}|X(x)|_{\g}\leq  1\right\}.
\end{equation}
In addition, it is known that $*X^{\flat}=dV_{\g}(X,\cdot)$, so now the definition of divergence in Riemannian manifolds implies that $d*X^{\flat}=\div_{\g}X\, dV_{\g}$. Also, the volume form $dV_{\g}$ satisfies that $\langle dV_{\g},\xi\rangle\equiv 1$ for any unit $n$-vector field $\xi$. Consequently,
\[
\curr{E}(d*X^{\flat})=\int_E\langle\div_{\g}X\,dV_{\g},\xi\rangle_x\,d\Hh^n_{\g}(x)
=\int_{E}\div_{\g}X\,d\Hh^n_{\g}.
\]
Here $d\Hh^n_{\g}$ denotes the Riemannian measure. As $E$ is countably $n$-rectifiable, in the study the last integral we can assume with no loss of generality that $E$ is an $n$-dimensional embedded $C^1$-submanifold of $\R^n$, cf.~\cite[\S11.1]{simon1983lectures}. In this case, let us note the pullback metric in $E$ is again $\g$, and that in the Euclidean ambient spaces there is no need to use local parametrizations and partitions of unity to compute the integration with respect to the Riemannian measure in $E$. From all of this, and the expression in coordinates of the divergence operator of $X=\sum_jX^j(\partial/\partial x^j)\in C^{\infty}_c(U;\R^n)$, we simply obtain
\begin{align*}
\curr{E}(d*X^{\flat})
&=\int_{E}\sqrt{\det \g}\;\div_{\g}X\, dx\\
&=\int_{E}\sum^n_{j=1}\frac{\partial}{\partial x^j}(\sqrt{\det \g}\, X^j)\, dx\displaybreak[3]\\
&=\int_{\R^n}\chi_E\, \div(\sqrt{\det \g}\; X)\,dx,
\end{align*}
where $\div(\cdot)=\div_{\bar{\g}}(\cdot)$ stands for the divergence in Euclidean coordinates. Defining now the vector field $Y:=\sqrt{\det \g}\, X\in C^{\infty}_c(U;\R^n)$, we get from $\g=\a^{\sigma}\bar{\g}$ that $\sqrt{\det \g(x)}=\a^{n\sigma/2}(x)$, and consequently
\begin{align*}
|X|_{\g}=\sqrt{\g_x(X,X)}=\a^{\sigma/2}|X|_{\bar{\g}}=\a^{(1-n)\sigma/2}(\a^{n\sigma/2}|X|_{\bar{\g}})
=\a^{(1-n)\sigma/2}|Y|_{\bar{\g}}.
\end{align*}
The computation of the mass $\M_{U,\g}(\partial\curr{E})$ in~\eqref{eqn:intermass} has been shown to be equivalent to
\begin{equation}
\sup\left\{\int_{\R^n}\chi_E\,\div Y\, dx:\, Y\in C^{\infty}_c(U;\R^n),\; \sup_{x\in U} \a^{-(n-1)\sigma/2}(x)|Y(x)|_{\bar{\g}}\leq  1\right\}.
\label{eqn:5}
\end{equation}
Finally, we observe that $\a^{-(n-1)\sigma/2}(x)|Y(x)|_{\bar{\g}}=\varphi^0_{\a}(x,Y(x))$, where $\varphi^0_{\a}(x,\cdot)$ is the polar (or dual) of the norm $\varphi_{\a}(x,\cdot):=\a^{(n-1)\sigma/2}(x)|\cdot|_{\bar{\g}}$. This fact along with definition~\eqref{eqn:defavar} allows us to identify the supremum in~\eqref{eqn:5} as an inhomogeneous variation of $\chi_E$ in $U\subset\R^n$:
\[
\int_U\a^{(n-1)\sigma/2}(x)|D\chi_E| ,
\]
thus completing the proof of Theorem~\ref{thm:massequiv}.
\end{proof}

\begin{rem}\label{rem:regularity}
The singularity estimates~\eqref{eqn:regularitySSA} mentioned in \S\ref{sec:notation} are a consequence of the regularity theory in~{\em\cite{schoen1977regularity,schoen1982new}} for co-dimension one rectifiable currents minimizing parametric elliptic functionals, with the choice of $\mathbf{F}(T)=\int_{\R^n}F(x,\nu^T(x))d\|T\|$ and $F(x,p)=\a(x)|p|$ {\em(}following their notation{\em)}, provided $\a\in C^2(\bar{\Omega})$ is a weight function bounded away from zero. However, if we focus simply in the question of existence of tangent cones at every point on the support of a mass-minimizing boundary current $T=\partial\curr{E}$ in the Riemannian setting, we can just invoke the identification in Theorem~\ref{thm:massequiv} above. Indeed, this also follows from applying what are now standard techniques {\em(}cf.~Federer~{\em\cite{federer1969geometric}} or Simon~{\em\cite{simon1983lectures})} in the regularity theory of integral mass-minimizing currents in the Euclidean space $(\R^n,\bar{\g})$ with $\bar{\g}(x)=\delta_{ij}dx^idx^j$. Given $x_0\in \spt T$ on a Riemannian manifold $(M,\g)$ we let $x=(x^1,\ldots, x^n)\in\R^n$ be normal coordinates for $M$ near $x_0$ with the origin $x=0$ corresponding to $x_0$ and with $T_{x_0}M$ identified with $\R^n$ via these coordinates {\em(}here $\g\in C^2$ is required{\em)}. Thus, in these coordinates $\g(x)=\g_{ij}(x)dx^idx^j$ with $\g_{ij}(0)=\delta_{ij}$ and $\partial \g_{ij}/\partial x^k(0)=0$ for all $i,j,k$. We can take homotheties $T_{\lambda}:=(\lambda^{-1})_{\#}T$ for $\lambda>0$ in terms of such coordinates, and $T_{\lambda}$ is mass-minimizing relative to the metric $\g_{ij}(\lambda x)dx^idx^j$. In light of the approximation for $\g$, Euclidean density estimates around $x=0$ translate into estimates for the mass ratio of the area in metric balls around $x_0$. A monotonicity formula for the density function is then available, which combined with the compactness theorem for locally rectifiable currents {\em(}in the Riemannian setting{\em)}~\emph{\cite[Thm.~5.5]{morgan2009geometric}} show that for any $\lambda_j\to 0^+$ the sequence $T_{\lambda_j}$ has a subsequence converging weakly to some current $C$ {\em(}i.e. the tangent cone at $x_0${\em)}, invariant under large homotheties. A detailed account of this argument is given in~\emph{\cite[pp.~5044]{morgan2003regularity}}.  A further recollection of results for currents in the Riemannian setting can be found in~\emph{\cite{hardt1985area}}. 
\end{rem}

\subsection{A proof of Theorem~\ref{thm:setmaxpple}}

The argument of this result in the Euclidean case (i.e. $\a(x)\equiv 1$) is given \cite[Thm.~2.2]{sternberg1992existence}. In what follows, we indicate how their proof goes through as well in our inhomogeneous setting given by a weight function $\a(x)$.
\smallskip

As in the statement of the theorem, consider sets $E_1$ and $E_2$ minimizing the $\a$-perimeter in an open set $U\subset \R^n$, and we assume that $x\in \partial E_1\cap\partial E_2\cap U$. Let us select $r>0$ small so that $B_r(x)\subset\subset U$ and observe that, in light of $\overline{\partial^* E_j}=\partial E_j$ (see e.g~\cite[\S4]{giusti1984minimal}), it follows that $x$ is in the closure of $(\partial^* E_j)\cap B_r(x)$, for both $j=1,2$.

According to Definition~\ref{def1}, each connected component $S_{j,m}$ of $(\partial^*E_j)\cap B_r(x)$  equipped with the induced orientation of $\R^n$, for each $j=1,2$ and $m\geq 0$, can be considered as an $(n-1)$-dimensional multiplicity one rectifiable current $T_{j,m}:=\curr{S_{j,m}}\lelbow B_r(x)$ in the ambient Riemannian manifold $\R^n$ endowed with the metric $\g(x)=\a^{2/(n-1)}(x)\delta_{ij}dx^idx^j$. The choice of the power $2/(n-1)$ in the weight function $\a$ guarantees that $\M_{B_r(x),\g}(\partial\curr{E_j})=\P_{\a}(E_j,B_r(x))$, as stated in Theorem~\ref{thm:massequiv}. In particular, since each $\partial E_j$ minimizes the $\a$-area in $B_r(x)$, then every $S_{j,m}$ is in fact a mass-minimizing current for the metric $\g$ considered above. This condition together with the non-degeneracy \eqref{ND} of $\a$ implies that any such component $S_{j,m}$ which intersects $B_{\rho/2}(x)$ for $\rho>0$ small, satisfies furthermore $\M_{B_{\rho}(x),\g}(S_{j,m})\geq c\rho^n$ for some constant $c>0$~(see e.g. \cite[\S5.1.6]{federer1969geometric}), so at most finitely many components of $(\partial^* E_j)\cap B_{\rho}(x)$ can intersect $B_{\rho/2}(x)$.

This argument shows as a matter of fact that $x$ lies in the closure of some connected component of $(\partial^* E_1)\cap B_{r}(x)$ and of $(\partial^* E_2)\cap B_{r}(x)$, say $S_{1,m_1}$ and $S_{2,m_2}$, respectively, where $m_1,m_2\geq 0$.  Let us observe that for fixed $j\in\{1,2\}$ and non-negative integers $m\neq m'$, the sets $S_{j,m}$ and $S_{j,m'}$ are disjoint by definition. Hence, an application of Theorem~\ref{thm:simonmaxpple} with $T_{j,m}=\curr{S_{j,m}}\lelbow B_r(x)$ ($\spt T_{j,m}=\overline{S_{j,m}}\subset\partial E_j$) shows, in fact, that $x$ is in the closure of {\em precisely one} component of $(\partial^* E_1)\cap B_{r}(x)$ and of $(\partial^* E_2)\cap B_{r}(x)$, respectively. Thus, the numbers $m_1,m_2$ are uniquely determined; because of this fact let us simplify notation by writing $C_j:=S_{j,m_j}$ for $j=1,2$, so that
\begin{equation}\label{eqn:strictmaxppleproof}
x\in \bar{C}_1\cap\bar{C}_2.
\end{equation}
The rest of the proof consists of showing that $\bar{C}_1=\bar{C}_2$. Once this is accomplished, the conclusion of Theorem~\ref{thm:setmaxpple} readily follows: as seen in the preceding paragraph, there are only finitely many components of $(\partial^* E_j)\cap B_r(x)$ that can intersect any compact subset of $B_r(x)$, so if $\rho>0$ is chosen small enough, then we have 
\[
\partial E_1\cap B_{\rho}(x)=\bar{C}_1\cap B_{\rho}(x)=\bar{C}_2\cap B_{\rho}(x)=\partial E_2\cap B_{\rho}(x).
\] 
Let us first assume that $C_1\cap C_2\neq \emptyset$. Using that $E_1\subset E_2$, we deduce that $C_1$ lies locally on one side of $C_2$ near each point of $C_1$. Since $C_j\subset \reg E_j$ for $j=1,2$ we know that $C_1, C_2$ can be locally represented as $C^2$-hypersurfaces around any contact point, whence the strict maximum principle Proposition~\ref{prop:strongmaxpple} together with the connectedness of $C_1$ and $C_2$ implies that $C_1=C_2$, and in particular $\bar{C}_1=\bar{C}_2$, which concludes the analysis of this case. 

For the second case let us suppose that $C_1\cap C_2=\emptyset$. We now invoke an argument of Simon from the proof of Theorem~\ref{thm:simonmaxpple} so as to characterize $C_1,C_2$ as the boundaries of sets $F_1,F_2$ which are nested (see~\cite[pp.~330]{simon1987strict}; here $F_j$  corresponds to his $E_j$). This will allow us to use Corollary~\ref{cor:simonsetmaxpple} to finish the  proof of Theorem~\ref{thm:setmaxpple}. We now recall the argument in~\cite{simon1987strict}. Let us write $K:=\bar{C}_1\cap\bar{C}_2\cap B_r(x)$ and observe that the regularity theory of $\a$-area minimizing sets~\eqref{eqn:regularitySSA} and our hypothesis $C_1\cap C_2=\emptyset$ imply that $\Hh^{n-2}(K)=0$. Also, we have already seen that the current $A_j:=\curr{C_j}\lelbow B_r(x)$ $(=T_{j,m_j})$ is a mass minimizing current in $B_r(x)$ for the metric $\g(x)=\a^{2/(n-1)}(x)\delta_{ij}dx^idx^j$ where $\reg A_j=C_j$ is connected and $\partial A_j=0$ in $B_r(x)$. Hence the decomposition theorem~\cite[\S4.5.17]{federer1969geometric} shows that $A_j=\partial\curr{F_j}\lelbow B_r(x)$ for some measurable set $F_j\subset B_r(x)$ for $j=1,2$. Up to an $\Hh^n$-negligible set we can consider $F_j$ to be one component of $B_r(x)\setminus \bar{C_j}$, so in particular $F_j$ is open, and connected, as $C_j$ is. In addition, we can arrange $C_1\cap F_2\neq\emptyset$ and $C_2\setminus \bar{F_1}\neq\emptyset$, by reversing orientations if necessary. This observation can be used together with the definition of $F_1$ and the fact that $F_2$ is open, to deduce that $F_1\cap F_2\neq\emptyset$. On the other hand, an application of the Poincare inequality~\cite[\S4.5.3]{federer1969geometric} together with the fact that $\Hh^{n-2}(K)=0$, and the connectedness of $C_1,C_2$, shows that $C_1\subset F_2\cup K$ and $C_2\subset (B_r(x)\setminus \bar{F}_1)\cup K$. From here we readily see that $F_1\subset F_2$, since otherwise we could choose a path $\gamma$ contained in $F_1$ connecting a point in $F_2$ to a point in $\bar{C}_2$ (recall $F_1$ is connected, and the definition of $F_2$), thus showing that $\bar{C}_2\cap F_1 \neq\emptyset$. Since $F_1$ is open, the latter observation yields $C_2\cap F_1\neq \emptyset$, but this contradicts the fact that $C_2\subset B_r(x)\setminus F_1$. Thus we have established that 
\begin{equation}\label{eqn:strictmaxppleproof2}
A_j=(\partial\curr{F_j})\lelbow B_r(x)\;\text{ for }\;j=1,2,\;\;\text{ with }\;F_1\subset F_2.
\end{equation}
Thus $A_j$ for each $j\in\{1,2\}$ is a mass-minimizing current in $B_r(x)$ satisfying~\eqref{eqn:strictmaxppleproof2} with $\spt A_j=\bar{C}_j$ connected, for which $\spt A_1\cap \spt A_2\cap B_r(x)\neq\emptyset$ due to~\eqref{eqn:strictmaxppleproof}. In light of these observations, we can apply Corollary~\ref{cor:simonsetmaxpple} to conclude that $\bar{C}_1=\bar{C}_2$. This not only finishes the analysis of our last case, but also the proof of Theorem~\ref{thm:setmaxpple}.\qed

\section{A geometric variational problem}

In this section we introduce an auxiliary variational problem for subsets of $\R^n$, which will be the cornerstone in the construction of a solution $\u$ to the original problem~\eqref{aLGP}. We follow the outline of the strategy introduced in~\cite{sternberg1992existence}. 

We will write henceforth $[a,b]:=g(\partial\Omega)$, and let us observe that $g\in C(\partial\Omega)$ admits a continuous extension on the complement of $\Omega$
\[
G\in BV(\R^n\setminus \bar{\Omega})\cap C(\R^n\setminus \Omega),\quad G=g\;\text{ on }\;\partial\Omega.
\] 
In fact, as shown in \cite[Thm.~2.16]{giusti1984minimal}, we can require that $\spt G\subset B_R(0)$ with $R$ chosen large enough so $\Omega\subset\subset B_R(0)$. Next, we introduce sets that will ensure our constructed solution satisfies the Dirichlet boundary condition $u=g$ on $\partial\Omega$. For each $t\in [a,b]$, we let
\begin{equation}\label{eqn:Lt}
\L_t:=\{x\in(\R^n\setminus \Omega): \, G(x)\geq t\}.
\end{equation}
The weighted co-area formula (Proposition~\ref{prop:wcoarea}) and the fact $G\in BV(\R^n\setminus \bar{\Omega})$ both imply that $\P_{\a}(\L_t\,,\R^n\setminus \bar{\Omega})<\infty$ for a.e. $t\in [a,b]$. However, the construction of the extension above can be made so that we also have 
$\P_{\a}(\L_t\,,\R^n\setminus \bar{\Omega})<\infty$ for all $t\in [a,b]$, cf.~\cite[Thm.~2.16]{giusti1984minimal}.

We remind the reader that we employ our convention~\eqref{eqn:convention} in defining $\L_t$. Using~\eqref{eqn:charaperim}, $\Hh^{n-1}(\partial\Omega)<\infty$, and the property $\max_{\partial\Omega}\a<\infty$, we deduce that for any $t\in [a,b]$
\begin{equation}\label{eqn:aperimequality}
\P_{\a}(\L_t\,,\R^n)=\P_{\a}(\L_t\,,\R^n\setminus \bar{\Omega})+\Hh^{n-1}\lelbow \a((\partial_M\L_t)\cap(\partial\Omega))<\infty.
\end{equation}
For each $t\in [a,b]$, consider the variational problems\\[-.25cm]
\begin{align}
&\inf\{\P_{\a}(E,\R^n):\, E\setminus \bar{\Omega}=\L_t\setminus \bar{\Omega}\},\label{start}\tag{$\star_t$}\\[.5em]
&\sup\{|E|:\; E \text{ is a solution of }\eqref{start}\}.\label{starprimet}\tag{$\star\star_t$}\notag\\[-0.2cm]\notag
\end{align}
\noindent Let us first observe that  
\begin{prop}\label{prop:existence}
The problem \eqref{starprimet} has a solution, for every $t\in [a,b]$. 
\end{prop}
In the remainder of this article we write $E_t$ for a solution to \eqref{starprimet}; a well defined object in light of Proposition~\ref{prop:existence}. Let us remark that our convention~\eqref{eqn:convention} ensures $E_t\setminus \bar{\Omega}=\L_t\setminus \bar{\Omega}$, and moreover, we observe that $\L_t$ need not be a closed set. 

\begin{proof}[Proof of Proposition~\ref{prop:existence}]
We argue the existence of a solution to~\eqref{start}, because this shows on the one hand that the admissible set  in~\eqref{starprimet} is non-empty, but also the existence of a solution to~\eqref{starprimet} can be derived along the same lines as for~\eqref{start}.

Let us write $m_t$ for the infimum value in~\eqref{start}. Observe $m_t<+\infty$, because $\L_t$ is admissible for~\eqref{start} with $\P_{\a}(\L_t\,,\R^n)<+\infty$ due to~\eqref{eqn:aperimequality}.  Recall that $\Omega\subset\subset B_R$ for $R$ large enough, and let $R$ represent this value throughout. Consider the auxiliary problem
\begin{equation}\label{eqn:auxpb}
\tilde{m}_t:=\inf\{\P_{\a}(F	, B_R):\; F\subset B_R,\; F\cap(B_R\setminus \bar{\Omega})=\L_t\cap(B_R\setminus \bar{\Omega})\}.
\end{equation}
As before, it is easy to see that $\tilde{m}_t<+\infty$. Let $\{F_j\}$ be a minimizing sequence for~\eqref{eqn:auxpb}, so $\P_{\a}(F_j,B_R)\to \tilde{m}_t$ as $j\to\infty$ and $F_j\subset B_R$, $F_j\cap(B_R\setminus \bar{\Omega})=\L_t\cap(B_R\setminus \bar{\Omega})$ for all $j$. The non-degeneracy~\eqref{ND} of the weight function, and the representation formula~\eqref{eqn:charaperim} of the $\a$-perimeter yield
\begin{align*}
\|\chi_{F_j}\|_{BV(B_R)}
&=|F_j|+\int_{B_R}|D\chi_{F_j}|\leq |B_R|+\frac{1}{\alpha}\int_{B_R}\a(x)|D\chi_{F_j}|\\
&\leq |B_R|+(\tilde{m}_t+1)/\alpha<\infty,
\end{align*} 
provided $j$ is large enough. Then, recalling the compact embedding $BV(B_R)\hookrightarrow L^1(B_R)$, we see that there exist a subsequence (denoted in the same way) and a set $\F_t\subset \R^n$ so that $\chi_{F_j}\to \chi_{\F_t}$ in $L^1(B_R)$, as $j\to\infty$. This fact combined with the lower semi-continuity of the $\a$-variation (Proposition~\ref{prop:lscaperim}), shows that
\begin{align*}
& |\F_t|\leq |B_R|,\\
& \P_{\a}(\F_t,B_R)\leq \liminf_{j\to\infty}\P_{\a}(F_j,B_R)=\tilde{m}_t.
\end{align*}
We will have shown that $\F_t$ solves~\eqref{eqn:auxpb}, once we argue that $\F_t$ is admissible for this problem. In view of our convention~\eqref{eqn:convention}, the admissibility follows if 
\[
|(\F_t\cap (B_R\setminus \bar{\Omega}))\Delta(\L_t\cap (B_R\setminus \bar{\Omega}))|=0.
\]
Since each $F_j$ is admissible for~\eqref{start}, the aforementioned $L^1$-convergence implies
\begin{align*}
|(\F_t\cap (B_R\setminus \bar{\Omega}))\Delta(\L_t\cap (B_R\setminus \bar{\Omega}))|
&=|(\F_t\Delta F_j)\cap(B_R\setminus \bar{\Omega})|\\
&\leq |(\F_t\Delta F_j)\cap B_R|=\int_{B_R}|\chi_{F_j}-\chi_{\F_t}|\,dx\xrightarrow[j\to\infty]{} 0.
\end{align*}
Now that we established the existence of a minimizer for the auxiliary problem, we easily get the one for the extended problem~\eqref{start}. Indeed, we put $\F^*_t:=(\L_t\setminus B_R)\cup \F_t$ and notice that $\F_t=(\L_t\cap (B_R\setminus \bar{\Omega}))\cup (\F_t\cap\bar{\Omega})$ shows that this set can be equivalently written
\[
\F^*_t=(\L_t\setminus \bar{\Omega})\cup(\F_t\cap\bar{\Omega}).
\]
In particular, any competitor $E$ of~\eqref{start} satisfies 
\[
\F^*_t\Delta E
=\F^*_t\Delta\{(\L_t\setminus \bar{\Omega})\cup (E\cap \bar{\Omega})\}
=(\F_t\cap\bar{\Omega})\Delta(E\cap\bar{\Omega})\subset\subset B_R,
\]
thus implying the inequality
\begin{align*}
\P_{\a}(E,\R^n)-\P_{\a}(\F^*_t,\R^n)
&=\P_{\a}(E,B_R)-\P_{\a}(\F^*_t,B_R)\\
&=\P_{\a}(E\cap B_R,B_R)-\P_{\a}(\F_t,B_R)\\
&\geq 0,
\end{align*}
where the last inequality comes from the fact that $E\cap B_R$ is admissible for~\eqref{eqn:auxpb}:
\[
(E\cap B_R)\cap (B_R\setminus \bar{\Omega})=(E\setminus \bar{\Omega})\cap B_R=
(\L_t\setminus \bar{\Omega})\cap B_R=\L_t\cap (B_R\setminus \bar{\Omega}).
\]
\end{proof}

\section{Construction of the solution}

In this section we study two crucial properties of the collection of sets $\{E_t:t\in [a,b]\}$. 
The construct of a solution $\u$ to \eqref{aLGP} will be carried out by basically identifying, up to an $\Hh^n$-negligible set, the superlevel set $\{\u\geq t\}$ with $E_t\cap\bar{\Omega}$, for almost all $t\in [a,b]$. The domains for which this construction is possible include bounded Lipschitz domains $\Omega$ with connected boundary, which in addition satisfy the barrier condition~\eqref{eqn:geomCiii}, stated in the introduction.

\smallskip

A key step in our development is the next property of the solution $E_t$ to~\eqref{starprimet}:

\begin{lem}[Boundary values]\label{lem:boundary}
Suppose $\partial\Omega$ satisfies the barrier condition with respect to the weight function $\a$. Then for any $t\in [a,b]$, 
\[
\partial E_t\cap \partial\Omega\subset \{g=t\}.
\]
\end{lem}

\begin{proof}[Proof of Lemma~\ref{lem:boundary}] Although we roughly follow the same line of argumentation as that found in~\cite{sternberg1992existence}, we make no use here of the PDE techniques invoked by those authors to reach a contradiction. Instead, we study topological properties of sets whose boundaries minimize the $\a$-area, while using the barrier condition~\eqref{eqn:geomCiii}.

Arguing by contradiction, let us suppose there exists a point $x_0\in(\partial E_t)\cap(\partial\Omega)$ with $g(x_0)<t$, and let $\e_0=\e_0(x_0)$ be the constant appearing in the barrier condition at $x_0$. We highlight that the existence of a minimizer $V_*$ of the variational problem~\eqref{eqn:geomCiii} can be established using the direct method in the calculus of variations, in the spirit of the proof of Proposition~\ref{prop:existence}. 

It will be convenient to write $B_{\e}:=B_{\e}(x_0)$ for this proof only, to simplify the notation. Let us observe that the continuity of $G$ together with the condition $g(x_0)<t$ guarantees  
\begin{equation}\label{eq:bdrylemmaaux1}
\L_t\cap B_{\e}=\emptyset,
\end{equation}
provided $0<\e<\e_0$ is taken small enough. We now fix such an $\e$. The mere fact that $E_t$ is a competitor for~\eqref{start} and~\eqref{eq:bdrylemmaaux1} yield the containment
\begin{equation}\label{eq:bdrylemmaaux2}
E_t\cap B_{\e}\subset \bar{\Omega}.
\end{equation}

Before continuing, we remark that $V_*$ and $E_t$ necessarily differ from each other near $x_0$,
despite the fact that under the contradiction hypothesis $g(x_0)<t$ the sets $E_t\cap B_{\e}$ and $V_*\cap B_{\e}$ are both contained in $\Omega$. In particular, we are assuming $x_0\in\partial E_t$ while $x_0\notin\partial V_*$ due to the barrier condition (in fact $\partial V_*\cap\partial\Omega\cap B_{\e}=\emptyset$.) It is this difference between $V_*$ and $E_t$ together with the minimizing property of $V_*$ respect to the $\a$-perimeter for inner variations of $\Omega$, that motivates us to utilize $V_*$ as a tool to construct a new set, $\E_*$. This set will have strictly less $\a$-perimeter than $E_t$, which will contradict the minimality of $E_t$ in~\eqref{start}, and corresponds to our main claim in the proof of Lemma~\ref{lem:boundary}. 

The latter will be achieved by proving that $\P_{\a}(E_t,\R^n)>\P_{\a}(\E_*,\R^n)$ where $\E_*:=(E_t\cap V_*\cap B_{\e})\cup (E_t\setminus B_{\e})$. To see this, let us first note from \eqref{eq:bdrylemmaaux1}-\eqref{eq:bdrylemmaaux2} that $\E_*$ is a competitor in \eqref{start}. Moreover, since $E_t=\E_*$ in $\R^n\setminus B_{\e}$ we can use the characterization of the $\a$-perimeter measure \eqref{eqn:charaperim}, to obtain
\begin{align}
\P_{\a}(E_t,\R^n)-\P_{\a}(\E_*,\R^n)
&=\P_{\a}(E_t,B_{\e})-\P_{\a}(\E_*,B_{\e})\notag\\
&=\Hh^{n-1}\lelbow\a(\partial^*E_t\cap B_{\e})-\Hh^{n-1}\lelbow\a(\partial^*(E_t\cap V_*)\cap B_{\e}),\label{eq:bdrylemmaaux3}
\end{align}
where
\begin{align*}
\Hh^{n-1}\lelbow\a(\partial^*E_t\cap B_{\e})
&=\Hh^{n-1}\lelbow\a(\partial^*E_t\cap V_*\cap B_{\e})\\
&\qquad\qquad+\Hh^{n-1}\lelbow\a((\partial^*E_t\setminus V_*)\cap B_{\e}),\\[0.2cm]
\Hh^{n-1}\lelbow\a(\partial^*(E_t\cap V_*)\cap B_{\e})
&=\Hh^{n-1}\lelbow\a(\partial^*E_t\cap V_*\cap B_{\e})\\
&\qquad\qquad+\Hh^{n-1}\lelbow\a(\partial^*V_*\cap E_t\cap B_{\e}),
\end{align*}
see Figure 1(a) below.

Applying these identities to \eqref{eq:bdrylemmaaux3} we obtain
\begin{equation}
\P_{\a}(E_t,\R^n)-\P_{\a}(\E_*,\R^n)
=\Hh^{n-1}\lelbow\a((\partial^*E_t\setminus V_*)\cap B_{\e})-\Hh^{n-1}\lelbow\a(\partial^*V_*\cap E_t\cap B_{\e}).\label{eq:bdrylemmaaux4}
\end{equation}

\begin{figure}[!h]
\centering
  \includegraphics[width=1.\textwidth]{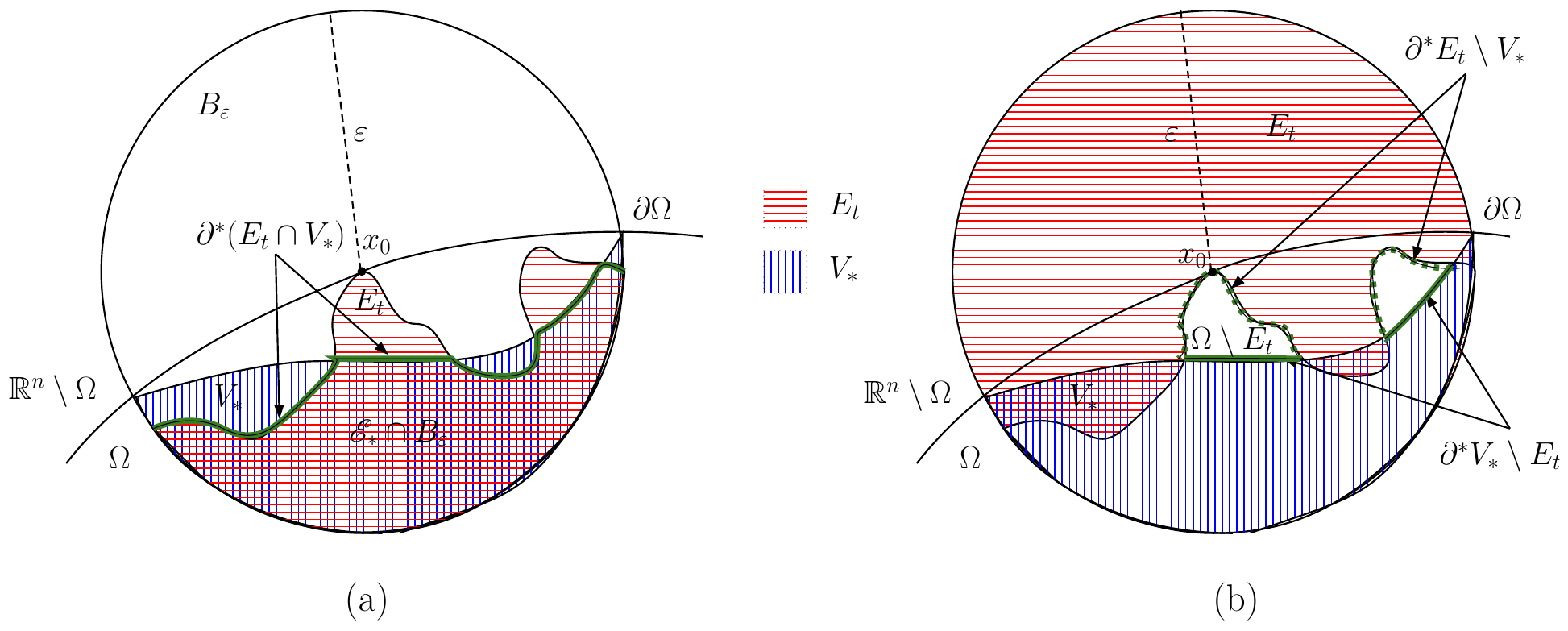}
\caption{\small Sketch of contradiction in the cases: (a) $g(x_0)<t$, and (b) $g(x_0)>t$.}
\end{figure}

On the other hand, let us observe that the set $\V:=V_*\cup (E_t\cap \Omega)$ is admissible for problem~\eqref{eqn:geomCiii}, and furthermore \eqref{eq:bdrylemmaaux2} shows that $E_t\cap B_{\e}=(E_t\cap\Omega)\cap B_{\e}$ $\Hh^n$-a.e. Consequently 
\begin{equation}
\V=([V_*\cup E_t]\cap B_{\e})\cup(V_*\setminus B_{\e}),\label{eq:bdrylemmaaux5}
\end{equation}
in light of our convention~\eqref{eqn:convention}. In particular, the statement $x_0\in \partial E_t$ can be made more precise by means of the barrier condition of $\partial\Omega$ for the minimizer $V_*$ and~\eqref{eq:bdrylemmaaux5}, to read now $x_0\in (\partial E_t\setminus \bar{V}_*)\subset \partial \V$. Put another way, $x_0\in \partial\V\cap\partial\Omega\cap B_{\e}$, so by virtue of the barrier condition once again we see that $\V$ cannot be a minimizer of (18). Hence, by the characterization~\eqref{eqn:charaperim} of the weighted perimeter and the minimality of $V_*$, we derive 
\begin{align*}
0<\P_{\a}(\V,\R^n)-\P_{\a}(V_*,\R^n)
&=\P_{\a}(\V,B_{\e})-\P_{\a}(V_*,B_{\e})\\
&=\Hh^{n-1}\lelbow\a(\partial^*(V_*\cup E_t)\cap B_{\e})-\Hh^{n-1}\lelbow\a(\partial^*V_*\cap B_{\e}),
\end{align*}
where we used that $\V=V_*$ in $\R^n\setminus B_{\e}$. This inequality can be exploited in light of the identities below:
\begin{align*}
\Hh^{n-1}\lelbow\a(\partial^*(V_*\cup E_t)\cap B_{\e})
&=\Hh^{n-1}\lelbow\a((\partial^*E_t\setminus V_*)\cap B_{\e})\\
&\qquad\qquad+\Hh^{n-1}\lelbow\a((\partial^*V_*\setminus E_t)\cap B_{\e}),\\[0.2cm]
\Hh^{n-1}\lelbow\a(\partial^*V_*\cap B_{\e})
&=\Hh^{n-1}\lelbow\a(\partial^*V_*\cap E_t\cap B_{\e})\\
&\qquad\qquad+\Hh^{n-1}\lelbow\a((\partial^*V_*\setminus E_t)\cap B_{\e}),
\end{align*}
to simply get
\begin{equation}
\Hh^{n-1}\lelbow\a(\partial^*V_*\cap E_t\cap B_{\e})
<\Hh^{n-1}\lelbow\a((\partial^*E_t\setminus V_*)\cap B_{\e}).\label{eq:bdrylemmaaux6}
\end{equation}
It immediately follows from~\eqref{eq:bdrylemmaaux4} and \eqref{eq:bdrylemmaaux6} that
\begin{align*}
\P_{\a}(E_t,\R^n)-\P_{\a}(\E_*,\R^n)>0,
\end{align*}
thus finishing the proof of the contradiction argument in the case $g(x_0)<t$. 

The other case where $g(x_0)>t$ is argued again by contradicting the minimality of $E_t$ in~\eqref{start}, nonetheless, for the sake of completeness we briefly discuss the proof. Indeed, the continuity of the extension $G$ of $g$ shows
\begin{equation}
(B_{\e}\setminus\bar{\Omega})\subset E_t\cap B_{\e}\quad\iff\quad (B_{\e}\setminus E_t)\subset\bar{\Omega}\cap B_{\e}. 
\label{eq:bdrylemmaaux7}
\end{equation}
The claim is that $\P_{\a}(E_t,\R^n)>\P_{\a}(\E_*,\R^n)$, if we now take $\E_{*}:=([E_t\cup (\Omega\setminus V_*)]\cap B_{\e})\cup (E_t\setminus B_{\e})$. Just as in the previous case we analyze the difference
\begin{align}
\P_{\a}(E_t,\R^n)-\P_{\a}(\E_*,\R^n)
&=\Hh^{n-1}\lelbow\a(\partial^*E_t\cap B_{\e})-\Hh^{n-1}\lelbow\a(\partial^*(E_t\cup (\Omega\setminus V_*))\cap B_{\e})\notag\\
&=\Hh^{n-1}\lelbow\a((\partial^*E_t\setminus V_*)\cap B_{\e})-\Hh^{n-1}\lelbow\a((\partial^*V_*\setminus E_t)\cap B_{\e}),\label{eq:bdrylemmaaux8}
\end{align}
see Figure 1(b) above.

Let us consider now a new auxiliary set $\V:=V_*\cup (\Omega\setminus E_t)$, and perform a similar analysis as before. Noting $x_0\in\partial E_t$, from \eqref{eq:bdrylemmaaux7} we get $x_0\in \partial(\Omega\setminus E_t)$ which, in addition to $x_0\notin \bar{V}_*$ (by the barrier condition), implies that $x_0\in \partial\V$. Since $\V$ is admissible in~\eqref{start}, the barrier condition once again yields
\begin{align*}
\P_{\a}(\V,\R^n)-\P_{\a}(V_*,\R^n)
&=\Hh^{n-1}\lelbow\a(\partial^*(V_*\cup(\Omega\setminus E_t))\cap B_{\e})
-\Hh^{n-1}\lelbow\a(\partial^*V_*\cap B_{\e})\\
&=\Hh^{n-1}\lelbow\a((\partial^*E_t\setminus V_*)\cap B_{\e})
-\Hh^{n-1}\lelbow\a((\partial^*V_*\setminus E_t)\cap B_{\e})\\
&>0,
\end{align*}
leading to the desired conclusion, in view of \eqref{eq:bdrylemmaaux8}. The proof of Lemma~\ref{lem:boundary} is now complete.
\end{proof}

\medskip

Continuing the study of the family $\{E_t: t\in [a,b]\}$ we now give a basic geometric description on how they are positioned inside of the domain $\Omega$. 
\begin{lem}\label{lem:strictcont}
Suppose $\partial\Omega$ satisfies the barrier condition with respect to $\a$. Then, for any $s,t\in [a,b]$ with $s<t$, there holds
\[
E_t\subset\subset E_s.
\]
\end{lem}

\begin{proof}[Proof of Lemma~\ref{lem:strictcont}]
The containment $E_t\subset E_s$ follows from the same argument as in \cite{sternberg1992existence}.  
Let us observe $E_s\cap E_t$ is a competitor with $E_t$ in \eqref{start},
\begin{align*}
(E_s\cap E_t)\setminus \bar{\Omega}
&=(E_s\setminus \bar{\Omega})\cap (E_t\setminus \bar{\Omega})\\
&=(\L_s\setminus \bar{\Omega})\cap (\L_t\setminus \bar{\Omega})=\L_t\setminus \bar{\Omega}.
\end{align*}
In a similar fashion, it can be readily seen $E_s\cup E_t$ is a competitor with $E_s$ in $(\star_s)$. It follows
\[
\P_{\a}(E_s\cap E_t,\R^n)\geq \P_{\a}(E_t,\R^n)\;\;\text{ and }\; \P_{\a}(E_s\cup E_t,\R^n)\geq \P_{\a}(E_s,\R^n).
\]
As the $\a$-perimeter satisfies Proposition~\ref{prop:ineq}, the above inequalities imply $\P_{\a}(E_s\cup E_t,\R^n)=\P_{\a}(E_s,\R^n)$. On the other hand, $E_s$ solves the problem $(\star\star_s)$ thus yielding $|E_s\cup E_t|=|E_s|$. We conclude $|E_t\setminus E_s|=0$, which in view of our convention~\eqref{eqn:convention} then yields
\begin{equation}\label{eqn:auxcontainmt}
E_t\subset E_s.
\end{equation}
It remains to show that this containment is actually strict. This method uses topological arguments along with techniques from geometric measure theory, and is an adaptation from the proof of this lemma in~\cite{sternberg1992existence}. Let us start noting
\begin{equation}\label{eqn:containment}
E_t\setminus \bar{\Omega}=\L_t\setminus \bar{\Omega}\subset\subset \L_s\setminus \bar{\Omega}=E_s\setminus \bar{\Omega},
\end{equation}
relative to the topology on $\Omega^c$. In addition, since $s<t$ we observe that Lemma~\ref{lem:boundary} implies 
\begin{equation}\label{eqn:boundary}
\partial E_t\cap \partial E_s\cap \partial\Omega=\emptyset.
\end{equation}
In consideration of \eqref{eqn:auxcontainmt}-\eqref{eqn:containment}-\eqref{eqn:boundary} we will prove the statement of this Lemma by showing 
\[
S:=\partial E_s\cap \partial E_t\cap \Omega=\emptyset.
\]
For this purpose let us assume on the contrary that $S\neq\emptyset$. The goal of the remainder of our proof is to verify the points below:
\begin{enumerate}[label={\rm(\Roman*)},itemsep=0pt,leftmargin=*]
\item\label{enum:topproofI} $S$ consists of the connected components of $\partial E_t$ that do not intersect $\partial\Omega$.
\item\label{enum:topproofII}  If $S'$ denote any connected component of $\reg(S)$, then $\bar{S'}$ has to intersect $\partial\Omega$.
\end{enumerate}
Using the density of $\reg(\partial E_t)$ in $\partial E_t$ (see \eqref{eqn:regularitySSA} in \S\ref{sec:notation}), we immediate conclude from~\ref{enum:topproofI}-\ref{enum:topproofII}  that $S$ must be empty, thus reaching a contradiction.

\smallskip

To argue~\ref{enum:topproofI}  first note that $S$ is open relative to $\partial E_t$, for if $x\in S$, then from $E_t\subset E_s$ and from the fact that both $\partial E_t$, $\partial E_s$ are $\a$-area minimizing in $\Omega$, we can apply the maximum principle Theorem~\ref{thm:setmaxpple} to conclude that $\partial E_t$ and $\partial E_s$ must agree on a neighborhood of $x$. On the other hand, $S$ is clearly closed relative to $\partial E_t$, so from~\eqref{eqn:boundary} we get that every connected component of $S$ is disjoint from $\partial\Omega$.

\medskip

The proof of~\ref{enum:topproofII}  is based on the fact that $\partial E_t$ is $\a$-area minimizing in $\Omega$. This is a rather general fact as given in the following

\begin{lem}[\!\!{\cite[Lem.~4.5]{jerrard2018existence}}]\label{lem:boundarycond}
Let $\Omega$ be a bounded Lipschitz domain with connected boundary, and assume $E\subset \R^n$ is $\a$-perimeter minimizing in $\Omega$. If $R$ is a connected component of $\reg(\partial E)\cap \Omega$, then $\bar{R}\cap \partial\Omega\neq\emptyset$. 
\end{lem}

The proof of Lemma~\ref{lem:boundarycond} is based on topological arguments in geometric measure theory. A full proof can be found in~\cite{jerrard2018existence}, where the authors argue by contradiction that if a set $E$ is $\a$-perimeter minimizing and admits some component of $\reg(\partial E)$ whose closure is disjoint from $\partial\Omega$, then $E$ could have not been $\a$-perimeter minimizing in the beginning.
\end{proof}

\smallskip

\section{Proof of the main Theorem~\ref{thm:main1}}

We are now in position to build up a continuous solution to the weighted least gradient problem \eqref{aLGP}.
For $t\in [a,b]$, we define
\begin{equation}\label{eqn:defAt}
A_t=\overline{(E_t\cap\Omega)}.
\end{equation}
Let us observe that $E_t$ is a closed set since every point of $\partial E_t$ is either a regular point of $\partial E_t$ or a point where the tangent cone exists (see~\S 2). This shows that $\partial E_t\subset \partial_M E_t$, and so our convention~\eqref{eqn:convention} then yields $\partial E_t\subset E_t$. The latter observation combined with the identity $\partial(U\cap V)\cap V^i=\partial U\cap V^i$ for any sets $U,V$ and the fact that $\Omega$ is open, show altogether
\[
A_t\cap\Omega=(\partial E_t\cap\Omega)\cup ((E_t)^i\cap\Omega)\subset E_t\cap\Omega,
\]
therefore $A_t\cap\Omega=E_t\cap\Omega$, for any $t\in[a,b]$. In addition, using the boundary value  Lemma~\ref{lem:boundary} and topological considerations we can see that for any $t\in [a,b]$, there hold
\begin{align}
& \{g>t\}\subset (E_t)^i\cap\partial\Omega\subset A_t\cap\partial\Omega,\label{eqn:Atcontainment1}\\[.2em]
& \overline{\{g>t\}}\subset A_t\cap\partial\Omega\subset 
[(E_t)^i\cup \partial E_t]\cap\partial\Omega\subset \{g\geq t\}.\label{eqn:Atcontainment2}
\end{align}
Finally, we observe that for any $s<t$ with $s,t\in [a,b]$
\begin{equation}\label{eqn:strictcontAt}
A_t\subset\subset A_s,
\end{equation}
relative to $\bar{\Omega}$. Indeed, topological considerations show that for a.e.~$t$, $\partial_{\bar{\Omega}}A_t\subset \partial E_t$, where $\partial_{\bar{\Omega}}$ denotes the topological boundary relative to the subspace topology of $\bar{\Omega}$. The validity of \eqref{eqn:strictcontAt} is then a mere consequence of Lemma~\ref{lem:boundary} and Lemma~\ref{lem:strictcont} combined.\\

The definition of our candidate for a solution is the one given below:
\begin{equation}\label{eqn:defnsoln}
\u(x):=\sup\{t\in [a,b]: \; x\in A_t\}.
\end{equation}
The next result asserts that $\u$ gives rise to a continuous function which meets the boundary condition in the strong sense.  

\begin{thm}\label{thm:constructionsoln}
The function $\u$ given in~\eqref{eqn:defnsoln} and the set $A_t$ defined in~\eqref{eqn:defAt} satisfy \begin{enumerate}[label={\rm(\roman*)},itemsep=0pt,leftmargin=*]
\item $\u=g$ on $\partial\Omega$.\\[-0.3cm]
\item $\u\in C(\bar{\Omega})$.\\[-0.3cm]
\item\label{enum:thmu:iii}$A_t\subset \{\u\geq t\}$ for all $t\in [a,b]$, and $|\{\u\geq t\}\setminus A_t|=0$ for a.e. $t\in [a,b]$.
\end{enumerate}
\end{thm}
The proof of Theorem 3.5 in \cite{sternberg1992existence} naturally carries over to prove our Theorem~\ref{thm:constructionsoln}, since it relies solely on \eqref{eqn:Atcontainment1}-\eqref{eqn:Atcontainment2}, plus some basic topological and analytic considerations.\\

Finally, let us now restate and provide a proof of our main theorem:\\[.2cm]
\noindent {\bf Theorem~\ref{thm:main1}.} 
{\em For $n\geq 2$, let $\Omega\subset\R^n$ be a bounded Lipschitz domain with boundary satisfying the barrier condition~\eqref{eqn:geomCiii}, and let $\a\in C^2(\bar{\Omega})$ be a non-degenerate weight function in the sense of~\eqref{ND}. Then for any boundary data $g\in C(\partial\Omega)$, the function $\u$ defined in~\eqref{eqn:defnsoln} is a continuous solution to 
\begin{equation}\label{eqn:aLGP}
\inf\left\{\int_{\Omega}\a(x)|Dv|: \,v\in BV(\Omega),\;\; v=g \text{ on }\partial\Omega\right\},
\end{equation}
where $v=g$ is understood in the sense of traces of $BV$-functions. Furthermore, the superlevel sets of $\u$ minimize the weighted perimeter measure $\P_{\a}(\cdot,\Omega)$ with respect to competitors meeting the boundary conditions imposed by $g$ on $\partial\Omega$.  
} 

\begin{proof}[Proof of Theorem~\ref{thm:main1}]
We follow the outline of the proof of Theorem 3.7 in~\cite{sternberg1992existence}. For any competitor $v\in BV(\Omega)$ of~\eqref{eqn:aLGP}, consider the extension $\bar{v}\in BV(\R^n)$ of $v$ with $\bar{v}=G$ in $\R^n\setminus\bar{\Omega}$, and let $F_t=\{\bar{v}\geq t\}$. It is sufficient to show that 
\begin{equation}\label{eqn:aperimineqomega}
\P_{\a}(E_t,\Omega)\leq \P_{\a}(F_t,\Omega)\;\;\text{ for a.e. }t\in [a,b],
\end{equation}
since the weighted co-area formula (cf. Proposition~\ref{prop:wcoarea}) would then yield
\[
\int_{\Omega}\a(x)|D\u|= \int^b_a\P_{\a}(E_t,\Omega)\,dt\leq \int^{+\infty}_{-\infty}\P_{\a}(F_t,\Omega)\,dt=\int_{\Omega}\a(x)|D\bar{v}|<\infty,
\]
where in the first identity we have used that $A_t\cap\Omega=E_t\cap\Omega$ for all $t\in[a,b]$ and Theorem~\ref{thm:constructionsoln}\ref{enum:thmu:iii}. Let us start arguing~\eqref{eqn:aperimineqomega} by noting for all $t\in [a,b]$ that $F_t\setminus \bar{\Omega}=\L_t\setminus \bar{\Omega}$, while the set $E_t$ minimizes the $\a$-perimeter amongst competitors satisfying this condition. Hence,
\[
\P_{\a}(E_t,\R^n)\leq \P_{\a}(F_t,\R^n),
\]
or equivalently,
\begin{align*}
\P_{\a}(E_t,\R^n\setminus \bar{\Omega})+\P_{\a}(E_t,\partial\Omega)&+\P_{\a}(E_t,\Omega)\\
&\leq \P_{\a}(F_t,\R^n\setminus \bar{\Omega})+\P_{\a}(F_t,\partial\Omega)+\P_{\a}(F_t,\Omega).
\end{align*} 
The characterization~\eqref{eqn:charaperim} of the $\a$-perimeter measure and the fact that $F_t=E_t$ on $\R^n\setminus \bar{\Omega}$ show that the above inequality reduces to 
\begin{equation}\label{eqn:aperimineqrn}
\Hh^{n-1}\lelbow\a(\partial^*E_t\cap\partial\Omega)+\P_{\a}(E_t,\Omega)\leq 
\Hh^{n-1}\lelbow\a(\partial^*F_t\cap\partial\Omega)+\P_{\a}(F_t,\Omega).
\end{equation} 
On the other hand, let us note that Lemma~\ref{lem:boundary} implies that the set $\partial E_t\cap\partial\Omega$ is $\Hh^{n-1}$-null for a.e. $t\in[a,b]$.  Indeed, $\{g=t\}$ is a $\Hh^{n-1}$-null set for all but countably many $t\in [a,b]$, since $\Hh^{n-1}(\partial\Omega)<\infty$. From this and~\eqref{eqn:bdryinclusion} we get $\Hh^{n-1}\lelbow\a(\partial^*E_t\cap\partial\Omega)=0$ for a.e. $t\in [a,b]$. Thereby, in light of \eqref{eqn:aperimineqrn}, we will have established~\eqref{eqn:aperimineqomega} once we prove 
\begin{equation}\label{eqn:auxeqn}
\Hh^{n-1}\lelbow\a(\partial^*F_t\cap\partial\Omega)=0.
\end{equation} 
This will be argued in the same way as before, once we are able to prove that
\begin{equation}\label{eqn:techineq2}
\partial_MF_t\cap\partial\Omega\subset \{g=t\}.
\end{equation}
Let us recall $g$ is the trace on $\partial\Omega$ of $v$ admissible in~\eqref{eqn:aLGP}, so for $\Hh^{n-1}$-a.e. $x\in\partial\Omega$,
\begin{equation}\label{eqn:integralineq}
\lim_{r\to0}\dashint_{B_r(x)\cap\Omega}|v(y)-g(x)|\,dy=0.
\end{equation}
(cf. \cite[\S 5.14]{ziemer1989weakly}). Thus in order to prove~\eqref{eqn:techineq2}, we consider any $x\in\partial\Omega$ as in~\eqref{eqn:integralineq} such that $x\in \{g<t\}$, say $g(x)=t-\delta$ with $\delta>0$. It follows that
\begin{align*}
0&=\lim_{r\to0}\frac{1}{|B_r(x)\cap\Omega|}\int_{B_r(x)\cap\Omega\cap\{v\geq t\}}|v(y)-g(x)|\,dy\\[0.2cm]
&\geq \delta\lim_{r\to0}\frac{|B_r(x)\cap\Omega\cap\{v\geq t\}|}{|B_r(x)\cap\Omega|}.
\end{align*}
Analogously, as $g$ is the trace of $\bar{v}\in BV(\R^n\setminus \Omega)$, the argument above shows 
\[
\lim_{r\to0}\frac{|B_r(x)\cap (\R^n\setminus \Omega)\cap \{\bar{v}\geq t\}|}{|B_r(x)\cap (\R^n\setminus \Omega)|}=0.
\]
These two identities above imply 
\[
\overline{\Theta}(\{\bar{v}\geq t\},x):=\limsup_{r\to0}\frac{|B_r(x)\cap \{\bar{v}\geq t\}|}{|B_r(x)|}=0,
\]
whence $x\notin \partial_M \{\bar{v}\geq t\}$, and so $\{g<t\}\subset\partial\Omega\setminus \partial_MF_t$. In a similar fashion, we can argue that $\{g>t\}\subset \partial\Omega\setminus \partial_MF_t$  by means of~\eqref{eqn:integralineq}, in light of
\[
\underline{\Theta}(\{\bar{v}\geq t\},x):=\liminf_{r\to0}\frac{|B_r(x)\cap \{\bar{v}\geq t\}|}{|B_r(x)|}\geq 1-\overline{\Theta}(\{\bar{v}<t\},x)=1,
\]
which holds for a.e $x\in\partial\Omega$ with $g(x)>t$. The conclusion is $\{g\neq t\}\subset \partial\Omega\setminus\partial_M F_t$, which finishes the proof of Theorem~\ref{thm:main1}.
\end{proof}

\medskip
\medskip

\bibliographystyle{amsplain}

\end{document}